\documentclass[12pt,leqno]{article}
\usepackage{amsmath, amsthm, amsfonts, amssymb, color}
\usepackage{mathrsfs}
\newtheorem{thm}{Theorem}[section]
\newtheorem{cor}[thm]{Corollary}
\newtheorem{lem}[thm]{Lemma}
\newtheorem{prp}[thm]{Proposition}

\newtheorem{exa}[thm]{Example}
\newtheorem{Remark}[thm]{Remark}
\theoremstyle{definition}

\newcommand{\scr}[1]{\mathscr #1}
\definecolor{wco}{rgb}{0.5,0.2,0.3}

\numberwithin{equation}{section}

\newcommand{\ua}{\uparrow}

\title{{\bf  Functional inequality on path space
over a non-compact Riemannian manifold}
}

\author{{\bf Xin Chen $^{a)}$ and Bo Wu $^{b)}$}\\
\footnotesize {$^{a)}$ Grupo de Fisica Matematica, Universidade de
Lisboa, Av Prof Gama Pinto 2, Lisbon 1649-003, Portugal}\\
\footnotesize {$^{b)}$ School  of Mathematical Sciences, Fudan
University, Shanghai 200433, China}\\
{\footnotesize E-mail: chenxin\_217@hotmail.com, wubo@fudan.edu.cn}}

\date{}

\begin{document}

\maketitle

\def\R{\mathbb R} \def\EE{\mathbb E} \def\Z{\mathbb Z} \def\ff{\frac} \def\ss{\sqrt}
\def\H{\mathbb H}
\def\dd{\delta} \def\DD{\Delta} \def\vv{\varepsilon} \def\rr{\rho}
\def\<{\langle} \def\>{\rangle} \def\GG{\Gamma} \def\gg{\gamma}
\def\ll{\lambda} \def\LL{\Lambda} \def\nn{\nabla} \def\pp{\partial}
\def\d{\text{\rm{d}}} \def\loc{\text{\rm{loc}}} \def\bb{\beta} \def\aa{\alpha} \def\D{\scr D}
\def\E{\scr E} \def\si{\sigma} \def\ess{\text{\rm{ess}}}
\def\beg{\begin} \def\beq{\beg}  \def\F{\scr F}
\def\Ric{\text{\rm{Ric}}}
\def\Var{\text{\rm{Var}}}
\def\Ent{\text{\rm{Ent}}}
\def\Hess{\text{\rm{Hess}}}\def\B{\scr B}
\def\e{\text{\rm{e}}} \def\ua{\underline a} \def\OO{\Omega} \def\b{\mathbf b}
\def\oo{\omega}     \def\tt{\tilde} \def\Ric{\text{\rm{Ric}}}
\def\cut{\text{\rm{cut}}} \def\P{\mathbb P} \def\ifn{I_n(f^{\bigotimes n})}
\def\fff{f(x_1)\dots f(x_n)} \def\ifm{I_m(g^{\bigotimes m})} \def\ee{\varepsilon}
\def\C{\scr C}
\def\M{\scr M}\def\ll{\lambda}
\def\X{\scr X}
\def\T{\scr T}
\def\A{\mathbf A}
\def\LL{\scr L}
\def\gap{\mathbf{gap}}
\def\div{\text{\rm div}}
\def\dist{\text{\rm dist}}
\def\cut{\text{\rm cut}}
\def\supp{\text{\rm supp}}
\def\Cov{\text{\rm Cov}}
\def\Dom{\text{\rm Dom}}
\def\Cap{\text{\rm Cap}}
\def\sect{\text{\rm sect}}\def\H{\mathbb H}

\begin{abstract}We prove the existence of the O-U Dirichlet form and the damped
O-U Dirichlet form on path space over a general
non-compact Riemannian manifold which is complete and stochastically complete.
We show  a
weighted log-Sobolev inequality for the O-U Dirichlet form and the (standard)
log-Sobolev inequality for the damped O-U Dirichlet form. In particular, the
Poincar\'e inequality (and the super Poincar\'e inequality) can be
established for  the  O-U  Dirichlet form on path space over a class of Riemannian manifolds with
unbounded Ricci curvatures. Moreover, we construct a
large class of quasi-regular local Dirichlet forms with unbounded random diffusion coefficients on
path space over a general non-compact manifold.
\end{abstract}

\vskip 2cm

\section{Introduction}

Suppose $M$ is a $n$-dimensional non-compact complete connected Riemannian
manifold, the Riemannian path space $C_{o,T}(M)$ over $M$ is defined by
$$C_{o,T}(M):=\{\gamma\in C([0,T];M):\gamma(0)=o\},$$
where $T$ is a positive constant and $o \in M$.
Let $d_M$ be the Riemannian distance on $M$, then $C_{o,T}(M)$ is a Polish
space under the uniform distance
$$d(\gamma,\sigma):=\displaystyle\sup_{t\in[0,T]}d_M(\gamma(t),\sigma(t)),\quad\gamma,\sigma\in C_{o,T}(M).$$

Let $O(M)$ be the orthonormal frame bundle over $M$, and let $\pi: O(M) \rightarrow M$ be the canonical projection.
Furthermore, we choose a standard othornormal basis $\{H_i\}_{i=1}^n$
of horizontal vector fields on $O(M)$ and consider the following SDE,
\begin{equation}\label{c1.0}
\begin{cases}
&\d U_t=\displaystyle\sum^n_{i=1}H_i(U_t)\circ\d W_t^i,\ \ t \in [0,\zeta),\\
& U_0=u_o,
\end{cases}
\end{equation}
where $u_o$ is a fixed orthonormal basis of $T_o M$, $W^1_t,\cdots,W_t^n$ are independent
Brownian motions on $\mathbb{R}$ and $\zeta$ is the maximal time of the solution.
Then  $X_t:=\pi(U_t),\ t \in [0,\zeta)$ is the Brownian motion on
$M$ with initial point $o$, and $U_{\cdot}$ is the (stochastic) horizontal lift along
$X_{\cdot}$. Throughout this paper, besides the completeness of $M$,
we assume further that $M$ is stochastically complete, i.e., $\zeta=\infty$, a.s..

Let $\mu_{o,T}$ be the distribution of $X_{\cdot}$ in the time interval $t\in [0,T]$, then $\mu_{o,T}$
is a probability measure on $C_{o,T}(M)$. From now on, we fix $o \in M$, $T=1$,
and for simplicity, we write $C_o(M)$ for $C_{o,1}(M)$ and
$\mu$ for $\mu_{o,1}$.
Let $\F C_b$ be the space of bounded Lipschitz continuous cylinder functions on $C_{o}(M)$, i,e, for
every $F \in \F C_b$, there exist some $m\geq1,
0<t_1<t_2\cdots<t_m\leq 1, f\in C^{Lip}_b(M^m)$ such that
$F(\gamma)=f\big(\gamma(t_1),\cdots,\gamma(t_m)\big)$, $\gamma \in C_{o}(M)$, where
$C^{Lip}_b(M^m)$ is the collection of bounded Lipschitz continuous functions on $M^m$.

Suppose $\mathbb{H}$ is the standard Cameron-Martin space for $C([0,1];\mathbb{R}^n)$, i.e.
\begin{equation*}
\begin{split}
&\mathbb{H}:=\Big\{h\in C([0,1]; \mathbb{R}^n)\Big|
h~ \text{is absolutely continuous}\\
&~~~~~~~~~~~~~~~~~~~~~~~~~~~~~~~~h(0)=0, \|h\|^2_{\mathbb{H}}:=\int_0^1| h'(s)|^2\d s<\infty\Big\},
\end{split}
\end{equation*}
where $h'(s)$ is the derivative with respect to the time variable $s$. In
fact, $\mathbb{H}$ is
a separable Hilbert space with the inner product $\<h,g\>_{\mathbb{H}}:=\int_0^1  \<h'(s) , g'(s)\>\d s,\
h,g\in \mathbb H.$
For any $F\in\F C_b$ with the form
$F(\gamma):=f\big(\gamma(t_1),\cdots,\gamma(t_m)\big)$ and any
$h\in\mathbb{H}$, we define the directional derivative $D_h F$ as following,
\begin{equation}\label{c1.1}
D_hF(\gamma):=\displaystyle\sum^m_{i=1}\<\nabla_i
f\big(\gamma(t_1),\cdots,\gamma(t_m)\big),U_{t_i}(\gamma)h(t_i)\>_{T_{\gamma(t_i)}M},
\end{equation}
where $\nabla_i$ is the (distributional) gradient operator for the $i$-th
component on $M^m$ and $U_{\cdot}(\gamma)$ is the horizontal lift along $\gamma(\cdot)$. Note that $D_hF$ is independent of the representation of
$F$, $\nabla_i f$ is
defined almost everywhere with respect to the
Riemannian volume measure, and the law of $\gamma(t), t \in (0,1]$ under $\mu$ is absolutely continuous
with respect to the Riemannian volume measure (see e.g. \cite{Hsu2}), so $D_h F$ is well defined, and it is defined
$\mu$-$a.s.$ on $C_{o}(M)$.
By Riesz representation theorem, there exists a gradient operator $DF(\gamma)\in \mathbb{H}$,
such that $\<DF(\gamma),h\>_{\mathbb{H}}=D_hF(\gamma),
 h\in\mathbb{H}$, $\gamma \in C_{o}(M)$. And for every $F \in \F C_b$ with the form above,
 it is easy to check that $DF$ has the following expression,
 $$(DF(\gamma))_s=\sum^m_{i=1}(s\wedge t_i)U_{t_i}(\gamma)^{-1}\nabla_i
f(\gamma),\quad \mu-a.s..$$
 Since $DF$ is bounded for every
$F \in \F C_b$, we can define a quadratic form as following,
\begin{equation}\label{c1.2}
\E(F,G):=\int_{C_{o}(M)}\<DF,DG\>_{\mathbb{H}}\d \mu, \ \ F,G\in\F C_b.
\end{equation}

It is well known that if the based manifold $M$ is compact or with bounded Ricci curvature, then the integration
by parts formula holds for $D_h$, hence the quadratic form $(\E, \F C_b)$ is closable. According to the theory of Dirichlet form, it is not difficulty to show that
the closed extension $(\E, \mathscr D(\E))$ is a conservative local Dirichlet form on $L^2(\mu):=L^2(C_o(M),\mu)$, which is usually called the O-U Dirichlet form. For the case $M$ compact, see \cite{CM}, \cite{D1}, \cite{EL}, \cite{ES}
\cite{FM}, \cite{Hsu1}, for the case $M$ non-compact with bounded Ricci curvature, see \cite{CHL}, \cite{Hsu3}.
In fact, in the integration by parts formula for $D_h$, a term depending on the Ricci curvature of the
based manifold appears,
to make such term integrable, it is natural to put some restrictions on the bound of the Ricci curvature. On the other hand,
since the horizontal lift $U_t$ is an isometry, without any condition on the bound of the
Ricci curvature, the quadratic form $(\E, \F C_b)$ is still well defined by (\ref{c1.2}). In
this article, we will show the following result about the closability of $(\E, \F C_b)$,
\begin{thm}\label{t1.1} The quadratic form $(\E, \F C_b)$ is closable on
 $L^2(\mu)$, and its closed
 extension $(\E, \D(\E))$ is a Dirichlet form on  $ L^2(\mu)$.
\end{thm}
In particular, only the completeness and the stochastic
completeness of the based manifold is needed in Theorem \ref{t1.1}. Under the same condition, i.e.
completeness and stochastic completeness of the based manifold, the existence of a
quasi-invariance flow on $C_o(M)$ was shown in \cite{HO}.

Provided the closability of $(\E, \F C_b)$, we can define its closed extension
$(\E, \mathscr{D}(\E))$ as the O-U Dirichlet form
on $L^2(\mu)$. A natural question is what functional inequality holds for the O-U Dirichlet form?
If the based manifold is compact or with
bounded Ricci curvature, the Poincar\'e inequality for the O-U Dirichlet form was first
shown in \cite{F}, after that the log-Sobolev inequality has also been established for the O-U Dirichlet form, see
e.g. \cite{AE}, \cite{CHL}, \cite{EL}, \cite{Hsu4}, \cite{Hsu3}. For a class of
based manifolds with unbounded Ricci curvatures, a weak Poincar\'e inequality
was shown in \cite{W} for the O-U Dirichlet form. In this article, we will study the functional inequality for
the O-U Dirichlet form on path space over a general non-compact manifold.

Let \begin{equation}\label{e15aa}\aligned&
K(\gamma):=\sup_{t \in [0,1]}\|\text{Ric}(\gamma(t))\|_{T_{\gamma(t)}M},\ \  \gamma \in C_o(M),\\&K_1(\gamma):=\inf_{t \in [0,1]}\inf_{v \in T_{\gamma(t)}M, |v|=1}\langle\text{Ric}(\gamma(t))v, v\rangle_{T_{\gamma(t)}M}
\ ,\ \ \  \gamma \in C_o(M).\endaligned
\end{equation}
For every $R \ge 0$, we define
\begin{equation}\label{e15c}
\begin{split}
& \tilde K(R):=\sup\Big\{\|\text{Ric}(x)\|_{T_x M}:\
x \in M,\ d_M(o,x)\le R\Big\},\ \ \\
&\tilde K_1(R)
:=\inf\Big\{\langle\text{Ric}(x)v, v\rangle_{T_x M}:\ x \in M,\ d_M(o,x)\le R
,\ v \in T_{x}M,\ |v|=1\Big\}.
\end{split}
\end{equation}

\begin{thm}\label{t1.3}
(1) The following weighted log-Sobolev inequality holds,
\begin{equation}\label{t1.3-1}
\mu(F^2\log F^2)\le \int_{C_o(M)}\big(4+K(\gamma)^2\e^{-K_1(\gamma)}\big)
\|D F\|_{\H}^2 \d \mu ,\ \ \  F \in \F C_{b,loc}.
\ \mu(F^2)=1,
\end{equation}

(2) Suppose
\begin{equation*}
\tilde K(s) \le c_1(1+s^{\delta_1}),\ \ \tilde K_1(s)\ge -c_2-\delta_2 \log(1+s),\ \ \ s>0,
\end{equation*}
for some non-negative constants $c_1,c_2,\delta_1,\delta_2$
satisfying $2\delta_1+\delta_2\le 2$, then the following  Poincar\'e inequality
\begin{equation*}
\mu(F^2)\le c_3\E(F,F)+\mu(F)^2,\ \ \ F \in \D(\E),
\end{equation*}
holds for some $c_3>0$.
\end{thm}

Theorem \ref{t1.3} is a combination of Theorem \ref{t4.1} and Corollary
\ref{cor} below. By our knowledge, it is the first result to show that the Poincar\'e inequality
holds for the O-U Dirichlet form on some path space whose based manifold may have unbounded Ricci curvature.
In particular, note that the right side of (\ref{t1.3-1}) may not be well defined for every
$F \in \F C_b$ without any condition on the curvature bound of the based manifold since the associated weighted
function may not be integrable, and
the weighted log-Sobolev inequality (\ref{t1.3-1}) holds for every $F \in \F C_{b,loc}$.
(see (\ref{e1}) below for the definition of $\F C_{b,loc}$)

As long as the based manifold is complete and stochastically
complete, we can construct the damped O-U Dirichlet form on $C_o(M)$,
see e.g. Example \ref{ex1} below. Moreover, the log-Sobolev inequality holds for the
damped O-U Dirichlet form with the corresponding constant to be $2$ (independent of
the curvature of the based manifold), see Theorem \ref{t4.1}.

In \cite{W}, a weak Poincar\'e inequality for the O-U Dirichlet form was shown under some conditions of $K_1$ and $K$,
but we are not sure
whether the weak Poincar\'e inequality is true if we only assume the based manifold is complete and stochastically complete.

Let
\begin{equation*}
\rho(\gamma):=\displaystyle\sup_{t\in[0,1]}d_M(\gamma(t),o),
\end{equation*}
we will show that for every $l \in C_0^{\infty}(\R)$, $l(\rho)\in
\D(\E)$, where $C_0^{\infty}(\R)$ denotes the set of smooth functions on $\R$ with compact supports.  Based on such property, we can construct
more general Dirichlet forms with diffusion coefficients,
which can be viewed as a generalization of those in
\cite{L} and \cite{WW1}.
In fact, let
\begin{equation}\label{e1}
\F C_{b,loc}:=\Big\{Fl(\rho): F \in \F C_{b},\ \ l \in C_0^{\infty}(\R)\Big\}
\end{equation}
be the collection of ``local" bounded Lipschitz continuous cylinder functions. Let
$\A: C_o(M)\times \H$ $\rightarrow \H$ be a measurable operator, such that,
\begin{enumerate}
 \item[(A1)] For $\mu$-$a.s.$ $\gamma \in C_o(M)$, $\A(\gamma):\H \rightarrow \H$ is a densely defined self-adjoint
operator with the domain $\D(\A(\gamma))$.
 \item[(A2)] For every $F \in \F C_{b,loc}$, $DF(\gamma)\in \D(\A(\gamma)^{\frac{1}{2}})$ for
 $\mu$-$a.s.$ $\gamma \in C_o(M)$ and
\begin{equation*}
\int_{C_o(M)}\big|\A(\gamma)^{\frac{1}{2}}\big(DF(\gamma)\big)\big|_{\H}^2 \d \mu <\infty.
\end{equation*}
\item[(A3)] For each $R>0$, there exists a constant $\vv(R)>0$ such that $\A(\gamma) \ge \vv(R)\mathbf{I}$
for $\mu$-$a.s.$ $\gamma\in C_o(M)$ satisfying $\rho(\gamma)\le R$,
where $\mathbf{I}$ denotes the identity operator.
\end{enumerate}
It is easy to see that under conditions (A1)-(A2), the following quadratic form
$(\E_{\A}, \F C_{b,loc})$ is well defined,
\begin{equation}\label{e2}
\E_{\A}(F,G):=\int_{C_o(M)}\big\langle \A(\gamma)^{\frac{1}{2}} DF(\gamma),
\A(\gamma)^{\frac{1}{2}} DG(\gamma)\big \rangle_{\H}d\mu,\ \ \  F,G \in \F C_{b,loc}.
\end{equation}
And in this article, it will be shown that if we assume (A1)-(A3),
then $(\E_{\A}, \F C_{b,loc})$ is closable, and its closed extension
$(\E_{\A}, \D(\E_{\A}))$ is a local Dirichlet form.

The quasi-regularity of a  Drichlet form, in particular
that on a infinite dimensional space (without locally compact property), implies the existence of the associated
Hunt process for the Dirichlet form. For the overall introduction of the properties of the quasi-regular
Dirichlet forms on infinite dimensional space, we refer the reader to \cite{MR}.
The quasi-regularity of the O-U Dirichlet form on path space over a compact manifold was first shown in
\cite{DR}.
And the quasi-regularity  of a class of Dirichlet forms with constant diffusion
coefficients was established in \cite{L}.
We also want to remark that if in condition (A2) above, we replace
the set $\F C_{b,loc}$ by $\F C_b$, and the constant $\vv(R)$ is independent of $R>0$ in (A3)
, then the quasi-regularity of such Dirichlet form $(\E_{\A}, \D(\E_{\A}))$ was shown in
 \cite{WW1}. See \cite{EM} for the case of the Dirichlet form on Finsler manifold,
 and see \cite{WW2} for the case of the Dirichlet form on free path space.
Another aim of this article is to prove the quasi-regularity of $(\E_{\A}, \D(\E_{\A}))$,
let the assumption (A2') be introduced as (\ref{e1a}) and (\ref{e1aa}) below, we can obtain the following result
\begin{thm}\label{t1.2}
Suppose (A1), (A2') and (A3) hold, then $(\E_{\A}, \D(\E_{\A}))$ is a quasi-regular
Dirichlet form.
\end{thm}

The article is organized as following, in the second section, we will prove the closability of
the quadratic form $(\E, \F C_b)$ and $(\E_{\A}, \F C_{b,loc})$.  In the third section,
we will show some functional inequalities for
the O-U Dirichlet form. In the fourth section, we will prove the
quasi-regularity for $(\E_{\A}, \D(\E_{\A}))$. In the Appendix, following the argument in \cite{TW},
we will prove a lemma needed in the proof of Theorem \ref{t1.1}.

\section{The closability of quadratic form}

In this section, we first show that the quadratic form $(\E, \F C_b)$ defined by (\ref{c1.1}) is closable.
The proof below is inspired by the
cut-off procedure for the Dirichlet form, see e.g. \cite[Proposition A.1]{BLW}, and the procedure of the conformal change
for the metric of the based manifold,
see e.g. \cite{TW}, \cite{W}, \cite{WW2}.

\begin{proof}[Proof of Theorem $\ref{t1.1}$]
(1) For every $R \ge 1$, let $B_R:=\{x\in M: d_M(x,o)\leq R\}$, and there exists a non-negative function
$f_R \in C_0^{\infty}(M)$ such that $f_R(x)=1$ for all $ x \in B_R$ and $M_R:=\{x \in M:\ f_R(x)>0\}$ is a
connected open set. We define a metric $\langle, \rangle_R$ on $M_R$ as
\begin{equation*}
\langle, \rangle_R:=f_R^{-2} \langle, \rangle,
\end{equation*}
where $\langle, \rangle$ is the Riemannian metric on $M$. Moreover, by
\cite[Section 2]{TW} and \cite[Lemma 3.4]{FWW}(see also Lemma \ref{l5.1} in the Appendix below),
we know $(M_R, \langle, \rangle_R)$ is a complete Riemannian manifold, and
\begin{equation}\label{e3}
K_R:=\sup_{M_R}\|\Ric^{(R)}\|_R <\infty,
\end{equation}
for every $R\ge 1$, where $\Ric^{(R)}$ denotes the Ricci curvature tensor on $(M_R, \langle, \rangle_R)$.
Hence $M_R$ is stochastically complete. We write $\mu_R$ for the distribution of
the Brownian motion on $C_o(M_R)$,  by the reference listed (see e.g. \cite{Hsu3}) in the introduction, there exists an
O-U Dirichlet form $(\E_R, \D(\E_R))$ on $L^2(\mu_R)$, such that,
\begin{equation*}
\E_R(F,F)=\int_{C_o(M_R)}\langle D_R F, D_R F\rangle_{\H} \d \mu_R, \ \  F \in \F C_b(M_R),
\end{equation*}
where $D_R$ denotes the (closed) gradient operator on $L^2(\mu_R)$.

In order to compare the O-U Dirichlet form $(\E_R, \D(\E_R))$ in different spaces $C_o(M_R)$,
we model them into the same probability space. Suppose that
$(\Omega, \F, \P)$ is a complete probability space, and $W_t$ is a $\R^n$-valued Brownian motion on this space.
We consider the SDE (\ref{c1.0}) on $M$,
\begin{equation}\label{e2a}
\begin{cases}
&\d U_t=\displaystyle\sum^n_{i=1}H_i(U_t)\circ\d W_t^i,\ \ t \in [0,1],\\
& U_0=u_o,
\end{cases}
\end{equation}
so $X_t:=\pi(U_t)$ is the Brownian motion on $M$, $U_{\cdot}$ is the horizontal lift along
$X_{\cdot}$.
Similarly, since $\langle, \rangle_R= \langle, \rangle,$ on $B_R$,
we can choose an orthonormal basis $\{H_{i,R}\}_{i=1}^n$  of horizontal vector fields
on $O(M_R)$ such that $H_{i,R}(u)=H_{i,m}(u)=H_i(u)$ for every $m\ge R$ when $u \in O(M_R)$ satisfies
$\pi(u) \in B_R$. Let $W_t$, $u_0$ be the same as that in (\ref{e2a}),
we consider the following SDE,
\begin{equation*}
\begin{cases}
&\d U_{t,R}=\displaystyle\sum^n_{i=1}H_{i,R}(U_{t,R})\circ\d W_t^i,\ \ t \in [0,1],\\
& U_{0,R}=u_o,
\end{cases}
\end{equation*}
so $X_{\cdot,R}:=\pi(U_{\cdot,R})$ is the Brownian motion on $M_R$, $U_{\cdot,R}$ is the horizontal lift
along $X_{\cdot,R}$ on $M_R$. Moreover, let $\tau_R:=
\inf\{t\ge 0: X_t \notin B_R\}$,
we have $U_{t,R}=U_{t,m}=U_t$ $\P$-$a.s.$ for every
$m \ge R, t \le \tau_R$.

Suppose $\{F_k\}_{k\ge 1}\subset \F C_b$ satisfy
\begin{equation}\label{e4}
\begin{split}
\lim_{k \rightarrow \infty}\mu(F_k^2)=0,\ \ \
\lim_{k,m \rightarrow \infty}\E(F_k-F_m,F_k-F_m)=0.\end{split}
\end{equation}
For every $R\ge 1$, we can find a
$l_R\in C^\infty_0(\R)$ such that
\begin{equation}\label{e0a}
l_R(r)=\begin{cases}
& 1,\quad\quad\quad\text{if}\ |r|\le R-1,\\
& \in [0,1], ~\text{if}\ R-1<|r|<R,\\
&0,\quad\quad\quad\text{if}\ |r|\ge R,
\end{cases}
\end{equation}
and $ \sup_{r \in \R}|l_R'(r)|\leq2$. Let
$d_R$ be the Riemannian distance on $M_R$ and
$$\rho_R(\gamma):=\sup_{t \in [0,1]}d_R(\gamma(t),o), \ \ \phi_R(\gamma):=l_R(\rho_R(\gamma)),\quad \gamma \in C_o(M_R).$$
We denote the gradient operator on $M_R$ by $\nabla^R$.
Since $|\nabla^R d_R|_R \le 1$ and the O-U gradient operator $D_R$ on $L^2(\mu_R)$ is closed due to
(\ref{e3}), as the same argument in
the proof of \cite[Lemma 2.2]{A} or \cite[Proposition 3.1]{RS}, we have
$\phi_R \in \D(\E_R)$, and $\|D_R \phi_R(\gamma)\|_{\H}\le 2$
for every $R \ge 1$ and $\mu$-$a.s.$ $\gamma \in C_o(M)$. Note that $l_R(r)=0$ if $r \ge R$ and $\rho_R(\gamma)=\rho(\gamma)$
for each $\gamma \in C_o(M_R) \subseteq C_o(M)$ satisfying
$\rho_R(\gamma)\le R$,  so we can extend $\phi_R$ to be defined in
$C_o(M)$ by $\phi_R(\gamma):=l_R(\rho(\gamma))$ for $\mu$-$a.s.$ $\gamma \in C_o(M)$.

Since $\{F_k\}_{k\ge 1}\subset \F C_b$, we may assume that for each $k\ge 1$,
$$F_k(\gamma)=f_k\big(\gamma(t_1),\cdots, \gamma(t_{j_k})\big),\quad \gamma\in C_o(M)$$
for some
$f_k \in C_b^{Lip}(M^{j_k}), j_k\ge 1$ and $0<t_1<\cdots <t_{j_k}\le 1$. Let $F_{k,R}(\gamma):=\phi_{R}(\gamma)F_k(\gamma)$,
since $\phi_R(\gamma) \neq 0$ only if
$\rho(\gamma)\le R$,
we can replace $f_k \in C_b^{Lip}(M^{j_k})$ by some $\tilde f_k \in C_b^{Lip}(M_R^{j_k})$ such that
$f_k(x)=\tilde f_k(x)$, $\forall \ x \in B_R^{j_k}$ in the definition of
$F_{k,R}$, then $F_{k,R}\big |_{C_o(M_R)}
\in \F C_{b,loc}(M_R) \subseteq \D(\E_R)$.
Note that $\rho(X_{\cdot})=\rho_R(X_{\cdot,R})\le R$ implies $X_{\cdot}
=X_{\cdot,R}$ and $U_{\cdot}
=U_{\cdot,R}$ $\P$-$a.s.$, hence it is easy to see that,
\begin{equation}\label{e4a}
D_R F_{k,R}(X_{\cdot,R})=\phi_R(X_{\cdot,R}) DF_k(X_{\cdot})+
F_k (X_{\cdot})D_R \phi_R(X_{\cdot,R}).
\end{equation}
Then we obtain, for every fixed $R \ge 1$,
\begin{equation}\label{c2.8}\aligned
&\E_R(F_{k,R}-F_{m,R}, F_{k,R}-F_{m,R})=\int
\|D_R F_{k,R}(X_{\cdot,R})-D_R F_{m,R}(X_{\cdot,R})\|_{\H}^2\d\P\\
&\le 2\int\phi_R^2(X_{\cdot,R}) \|DF_k(X_{\cdot})-DF_m(X_{\cdot})\|_{\H}^2\d\P\\&
~~~+2\int \|D_R\phi_R(X_{\cdot,R})\|_{\H}^2 |F_k(X_{\cdot})-F_m(X_{\cdot})|^2\d \P\\
&\le 2\int \|DF_k(X_{\cdot})-DF_m(X_{\cdot})\|_{\H}^2\d\P+4\int |F_k(X_{\cdot})-F_m(X_{\cdot})|^2\d \P\\
&=2\E(F_k-F_m,F_k-F_m)+4 \mu(|F_k-F_m|^2),
\endaligned
\end{equation}
where in the second inequality above, we use the property that $\phi_R \le 1$ and
$\|D_R \phi_R(\gamma)\|_{\H}\le 2$.
According (\ref{e4}), we have
\begin{equation}\label{e5}
\lim_{k,m\rightarrow\infty} \E_R(F_{k,R}-F_{m,R}, F_{k,R}-F_{m,R})=0.
\end{equation}
Note that by (\ref{e4}) and
as the same procedure above, it is not
difficult to check that,
\begin{equation}\label{e5a}
\lim_{k,m\rightarrow\infty}\mu_R(|F_{k,R}- F_{m,R}|^2)
\le \lim_{k,m\rightarrow\infty}\mu(|F_{k}- F_{m}|^2)=0.
\end{equation}
As mentioned earlier, $(\E_R, \D(\E_R))$ is closed due to (\ref{e3}), by (\ref{e5})
and (\ref{e5a}), we derive for every fixed $R\ge 1$,
\begin{equation}\label{e6}
\lim_{k\rightarrow\infty}\E_R(F_{k,R},F_{k,R})=0.
\end{equation}
We define $\mathbf{B}_R \subseteq C_o(M)$
by
\begin{equation}\label{e6aa}
\mathbf{B}_R:=\{\gamma \in C_o(M):\ \rho(\gamma)\le R\}.
\end{equation}
For every $k,m,R\ge 1$,
\begin{equation}\label{e6a}
\aligned
\E(F_k,F_k)&=\int\|D F_k (X_{\cdot})\|_{\H}^2\d\P=
\int \|D F_k(X_{\cdot})-D_R F_{k,R}(X_{\cdot,R})+D_R F_{k,R}(X_{\cdot,R})\|_{\H}^2 \d \P\\
&\le 2\int \|D F_k(X_{\cdot})-
D_R F_{k,R}(X_{\cdot,R})\|_{\H}^2 \d \P+ 2\E_R(F_{k,R},F_{k,R})\\
& \leq 4\int (1-\phi_R(X_{\cdot,R}))^2\|DF_k(X_{\cdot})\|_{\H}^2\d\P\\&~~~~~~~~~~~+
4\int \|D_R \phi_R(X_{\cdot,R})\|_{\H}^2F_k^2(X_{\cdot})\d \P+2\E_R(F_{k,R},F_{k,R})\\
&\le 4\int_{\mathbf{B}^c_{R-1}}\|DF_k(\gamma)\|_{\H}^2\d \mu+8\int F_k^2(\gamma)\d \mu+2\E_R(F_{k,R},F_{k,R})\\
&\le 8\int_{\mathbf{B}^c_{R-1}}\|DF_m(\gamma)\|_{\H}^2\d \mu+8\E(F_k-F_m,F_k-F_m)\\&
~~~~~~~~~~~~~~~~~~~
+8\int F_k^2(\gamma)\d \mu+2\E_R(F_{k,R},F_{k,R}),
\endaligned
\end{equation}
where in the second inequality above, we use (\ref{e4a}),
the third inequality is due to
the property that $\|D_R \phi_R(\gamma)\|_{\H}\le 2$ and
$\phi_R(X_{\cdot,R}) \neq 1$
only if $\rho_R (X_{\cdot,R})>R-1$, thus $\rho (X_{\cdot})>R-1$, and
the complement of $\mathbf{B}_R$ is denoted by $\mathbf{B}_R^c$.
According to (\ref{e4}) and (\ref{e6}), we obtain for every
fixed $R,m \ge 1$,
\begin{equation*}
\begin{split}
&\limsup_{k\rightarrow \infty}\E(F_k,F_k)\le
8\int_{\mathbf{B}^c_{R-1}}\|DF_m(\gamma)\|_{\H}^2\d \mu+8\limsup_{k \rightarrow \infty}\E(F_k-F_m,F_k-F_m),
\end{split}
\end{equation*}
and in the above inequality, first let $R \rightarrow \infty$ then $m \rightarrow \infty$, we get
\begin{equation*}
\limsup_{k \rightarrow \infty}\E(F_k,F_k)=0,
\end{equation*}
hence $(\E, \F C_b)$ is closable. Let $(\E, \D(\E))$ be the closed extension of
of $(\E, \F C_b)$, it is easy to show the contraction property of $(\E, \D(\E))$, for example, just repeating
the step (b) in the proof of \cite[Proposition 2.1]{WW1}. So we have proved
$(\E, \D(\E))$ is a symmetric Dirichlet form.
\end{proof}

\begin{lem}\label{l2.1}
For every $l \in C_0^{\infty}(\R)$, $l(\rho) \in \D(\E)$, and
 \begin{equation}\label{t2.1.1}
 \|Dl(\rho(\gamma))\|_{\H}\le \sup_{r \in \R}|l'(r)|,\ \ \ \ \ \mu- a.s. \ \gamma \in C_o(M).
 \end{equation}
\end{lem}

\begin{proof}
We follow the argument in the proof of \cite[Lemma 2.2]{A} or \cite[Proposition 3.1]{RS}.
We take a countable dense subset $\{t_i\}_{i=1}^{\infty}$ in $(0,1]$, and define
\begin{equation}\label{c*}
\rho^m(\gamma):=\sup_{1\le i \le m}d_M(\gamma(t_i),o),\ \
\phi^m(\gamma):=l(\rho^m(\gamma)),\quad \gamma \in C_o(M).
\end{equation}
It is obvious that $\phi^m:=l(\rho^m) \in \F C_b$. Note that
$\rho^m=g\Big(\big(d_M(\gamma(t_1),o),\dots,d_M(\gamma(t_m),o)\big)\Big)$, where
$$g(s):=\max_{1\le i \le m}s_i, \quad s=(s_1,\dots,s_m)\in \R^m,$$
since $g(s)$ is a Lipschitz continuous function on $\R^m$
with Lipschitz constant $1$, and $|\nabla d_M(o,x)|\le 1$, then for every $m \ge 1$, we have
\begin{equation*}
\|D \phi^m(\gamma)\|_{\H}\le \sup_{r \in \R}|l'(r)|,\ \ \ \mu-a.s. \ \gamma \in C_o(M).
\end{equation*}
By the Banach-Saks property, there exists a subsequence $\{\phi^{m_i}\}_{i=1}^{\infty}$ of $\{\phi^{m}\}$ such that
for $S_N:=\frac{1}{N}\sum_{i=1}^N \phi^{m_i}$, $\{D S_N\}_{N=1}^{\infty}$ is convergent in $L^2(\mu)$. Since
$\lim_{N \rightarrow \infty}\mu(|S_N-l(\rho)|^2)=0$, and $(\E, \D(\E))$ is closed, we obtain
$l(\rho) \in \D(\E)$ and (\ref{t2.1.1}) holds.
\end{proof}

We call $(\E, \D(\E))$ the O-U Dirichlet form on $L^2(\mu)$. Let $\F C_{b,loc}$ be
defined by (\ref{e1}), by Lemma \ref{l2.1},
we know $\F C_{b,loc}\subseteq \D(\E)$. Hence
the quadratic form $(\E_{\A}, \F C_{b,loc})$ is well defined by (\ref{e2}).
Furthermore, inspired
by \cite[Theorem 2.2]{L} and \cite[Proposition 2.1]{WW1}, we can show the closability
of $(\E_{\A}, \F C_{b,loc})$.

\begin{prp}\label{p2.1}
 Suppose (A1), (A2) and (A3) hold. The quadratic form $(\E_{\A},$ $ \F C_{b,loc})$ is closable on
 $L^2(\mu)$, and its closed
 extension $(\E_{\A}, \D(\E_{\A}))$ is a Dirichlet form.
\end{prp}

\begin{proof}
It is not difficult to show that $\F C_{b,loc}$ is dense in
$L^2(\mu)$ since $\F C_b$ is dense. Suppose $\{F_k\}_{k\ge 1}\subset \F C_{b,loc}$ satisfy
\begin{equation}\label{e7}
\begin{split}
\lim_{k \rightarrow \infty}\mu(F_k^2)=0,\ \ \
\lim_{k,m \rightarrow \infty}\E_{\A}(F_k-F_m,F_k-F_m)=0.
\end{split}
\end{equation}
Let $X_{\cdot,R}$, $X_{\cdot}$, $\phi_R$, $F_{k,R}$, $\mathbf{B}_R$ be the same terms as that in the proof of
Theorem \ref{t1.1}. From (\ref{e7}), we know $\{\A^{\frac{1}{2}}DF_k\}_{k=1}^{\infty}$
is a Cauchy sequence in $L^2(C_o(M)\rightarrow \H;\mu)$, hence there exists a
$\Phi \in L^2(C_o(M)\rightarrow \H;\mu)$, such that,
\begin{equation}\label{e8}
\lim_{k \rightarrow \infty}\int \big\|\A^{\frac{1}{2}}DF_k-\Phi\big\|_{\H}^2 \d\mu=0.
\end{equation}
In order to prove the closability of $\E_{\A}$, it suffices to show $\Phi=0$.

By assumption (A3), $\A^{-\frac{1}{2}}(\gamma)$ is a bounded operator on $\A(\gamma)^{\frac{1}{2}}
\big(\D(\A(\gamma)^{\frac{1}{2}})\big)$ with
$||\A(\gamma)^{-\frac{1}{2}}||\le \frac{1}{\sqrt{\vv(R)}}$ for $\mu$-$a.s.$ $\gamma\in \mathbf{B}_R$,
as the same argument for (\ref{c2.8}), we obtain
\begin{equation}\label{e7a}
\aligned
&\E_R (F_{k,R}-F_{m,R}, F_{k,R}-F_{m,R})\\&\le 2\int_{\mathbf{B}_R} \|
\mathbf{A}^{-\frac{1}{2}}\mathbf{A}^{\frac{1}{2}}(DF_k(X_{\cdot})-DF_m(X_{\cdot}))\|_{\H}^2\d\P
+4 \mu(|F_k-F_m|^2)\\&\le\frac{2\E_{\A}(F_k-F_m,F_k-F_m)}{\vv(R)}
+4\mu(|F_k-F_m|^2),
\endaligned
\end{equation}
hence from (\ref{e7}), $\lim_{k,m \rightarrow \infty}\E_R (F_{k,R}-F_{m,R}, F_{k,R}-F_{m,R})=0$,
then (\ref{e6}) is still true due to the closability of $(\E_R, \D(\E_R))$.

Note that $D_R F_{k,R}(X_{\cdot,R})=D F_k(X_{\cdot})$ for $\P$-$a.s.$ $\omega \in \Omega$ such that
$\rho(X_{\cdot})\le R-1$,
by (\ref{e6}) and (\ref{e8}), for every
$R\ge 1$, taking a subsequence if necessary (the subsequence may depend on
$R$),
\begin{equation}\label{e7aa}
\begin{split}
& \lim_{k \rightarrow \infty}\|DF_k(X_{\cdot})\|_{\H}=0,\\
&\lim_{k \rightarrow \infty}\|\mathbf{A}(X_{\cdot})^{\frac{1}{2}}(DF_k(X_{\cdot}))-
\Phi(X_{\cdot})\|_{\H}=0,\  \P-\ a.s.\ \omega \in \Omega\ \text{with}\ \rho(X_{\cdot})\le R-1.
\end{split}
\end{equation}
Since $\A(X_{\cdot})^{\frac{1}{2}}$ is closed, from (\ref{e7aa}) we know for
every $R\ge 1$,
$\Phi(X_{\cdot})=0$ for $\P$-$a.s.$ $\omega \in \Omega$ with $\rho(X_{\cdot})\le R-1$.  Note that
$R$ is arbitrary, we have $\Phi(\gamma)=0$ for $\mu$-$a.s.$ $\gamma \in C_o(M)$.

Let $(\E_{\A}, \D(\E_{\A}))$ be the closed extension of $(\E_{\A}, \F C_{b,loc})$, as the same argument in the step
(b) in the proof of \cite[Proposition 2.1]{WW1}, we can show the contraction property
of $\E_{\A}$, hence $(\E_{\A}, \D(\E_{\A}))$ is a Dirichlet form.
\end{proof}

For a suitable choice of $\A$, we will give the following example of $(\E_{\A}, \D(\E_{\A}))$,
which is the damped O-U Dirichlet form studied in
\cite{EL}, \cite{FM} when the based manifold is compact.
\\

\begin{exa}\label{ex1}
(Damped O-U Dirichlet form)
\end{exa}

As in \cite{EL} and \cite{FM},
we can define $\A^{\frac{1}{2}}(\gamma)$ pointwise.
For every $\gamma \in C_o(M)$ and $0\le t \le s \le 1$, let
$\Phi_{t,s}(\gamma) \in L(\R^n ; \R^n)$ be the solution of the following linear ODE,
\begin{equation*}
\frac{\d \Phi_{t,s}(\gamma)}{\d s}=-\frac{1}{2}\text{Ric}_{U_{\cdot}(\gamma)}^{\sharp}\big(\gamma(s)\big)\Phi_{t,s}(\gamma),
\ \  \Phi_{t,t}=\mathbf{I},\ 0\le t\le s \le 1,
\end{equation*}
where $U_{\cdot}(\gamma)$ is the horizontal lift along $\gamma$, and
$\text{Ric}_{U_{\cdot}(\gamma)}^{\sharp}\big(\gamma(s)\big)\in L(\R^n;\R^n)$ is defined by
$\langle \text{Ric}_{U_{\cdot}(\gamma)}^{\sharp}\big(\gamma(s)\big)a, b\rangle $
$=\big \langle \text{Ric}(\gamma(s))\big(U_s(\gamma)a\big), U_s(\gamma)b\big\rangle_{T_{\gamma(s)}M} $ for
every $a,b \in \R^n$. And we write $\Phi_s(\gamma):=\Phi_{0,s}(\gamma)$
for simplicity. Let $K(\gamma)$ and $K_1(\gamma)$ be defined by (\ref{e15aa}),
it is not difficult to see that
\begin{equation}\label{e14}
\|\Phi_{t,s}(\gamma)\|^2\le \e^{-K_1(\gamma)},\quad \gamma\in C_o(M).
\end{equation}

For every $\gamma \in C_o(M)$, we define $\hat \A(\gamma):\H \rightarrow \H$ as following,
\begin{equation}\label{e14a}\aligned
&\big(\hat \A(\gamma) h\big)(t)\\&=
h(t)-\frac{1}{2}\int_0^t \big(\Phi_r(\gamma)^{*}\big)^{-1}\int_r^1
\Phi_s(\gamma)^{*}\text{Ric}_{U_{\cdot}(\gamma)}^{\sharp}\big((\gamma(s))\big)
h'(s)ds,\ h \in \H.\endaligned
\end{equation}
where $\Phi_r(\gamma)^{*}$ denotes the adjoint operator of $\Phi_r(\gamma)$.
Note that
$\big(\Phi_r(\gamma)^{*}\big)^{-1}\Phi_s(\gamma)^{*}=\Phi_{r,s}(\gamma)^{*}$,
from (\ref{e14}) we know,
\begin{equation}\label{e15a}
\|\hat \A(\gamma)\|^2 \le 2\bigg(1+\frac{K(\gamma)^2 \e^{-K_1(\gamma)}}{4}\bigg),\
\ \ \gamma \in C_o(M),
\end{equation}
thus $\hat \A(\gamma)$ is a bounded operator.
Let $\A(\gamma):=(\hat \A(\gamma))^*\hat \A(\gamma)$, then
assumption (A1), (A2) is true for $\A$.

On the other hand,
\begin{equation}\label{e15}
\langle \hat \A(\gamma) h_1, h_2\rangle_{\H}:=
\langle h_1, \hat h_2(\gamma)\rangle_{\H},\ \  h_1,h_2 \in \H,
\end{equation}
where $\hat h_2(t,\gamma):=\Phi_t(\gamma)\int_0^t
\Phi_s(\gamma)^{-1} h_2'(s)ds $.

Let $\tilde {\mathbf A}(\gamma)$ be defined by,
\begin{equation*}
\langle \tilde {\mathbf A}(\gamma) h_1, h_2\rangle_{\H}:=
\langle h_1, \tilde h_2(\gamma)\rangle_{\H},\ \  h_1,h_2 \in \H,
\end{equation*}
where $\tilde h_2(\cdot,\gamma)\in \H$ and $\tilde h_2(t,\gamma):=
\int_0^t \Phi_s(\gamma)\frac{d}{ds}\big((\Phi_s(\gamma))^{-1}h_2(s)\big)ds$ is the solution to the following equation
$$ \tilde h_2'(t,\gamma)=\frac{1}{2}
\text{Ric}_{U_{\cdot}(\gamma)}^{\sharp}\big(\gamma(t)\big)h_2(t)
+h_2'(t),\ \ \ \tilde h_2(0,\gamma)=0.$$
Then $\tilde {\mathbf A}(\gamma)$ is a bounded operator
on $\H$ with
$$\|\tilde {\mathbf A}(\gamma)\|^2\le 2\bigg(1+\frac{K(\gamma)^2}{4}\bigg),\
\ \ \gamma \in C_o(M).$$
Furthermore, by (\ref{e15}), it is easy to show $\tilde {\mathbf A}(\gamma)\hat\A(\gamma)
=\mathbf{I}$, which implies that (A3) holds for $\A$ with
$$\vv(R)=\sup_{\gamma \in \mathbf{B}_R}
2\bigg(1+\frac{K(\gamma)^2}{4}\bigg).$$
By Theorem \ref{t1.1} and Theorem \ref{t1.2},
$(\E_{\A}, \F C_{b,loc})$ is closable, and its closure
$(\E_{\A}, \D(\E_{\A}))$ is a Dirichlet form.

We also want to remark that without any restriction on the bound of Ricci curvature of
$M$, $\E_{\A}$ may not be well defined on $\F C_b$ since $\A^{\frac{1}{2}}$ may not be integrable.
Hence the domain $\D(\E_{\A})$ may not be equal to $\D(\E)$, the domain of the O-U Dirichlet form, which is
different from the case in \cite{EL}, \cite{FM}, where the based manifold is compact.

\section{Functional Inequalities}

Through this section, let $\mathbf{\Lambda}(\gamma)$ be the operator $\A(\gamma)$ defined in
Example \ref{ex1}, so $(\E_{\mathbf{\Lambda}},\D(\E_{\mathbf{\Lambda}}))$ is the damped O-U Dirichlet form
on $L^2(\mu)$. We first show that the log-Sobolev inequality still holds for
$(\E_{\mathbf{\Lambda}},\D(\E_{\mathbf{\Lambda}}))$. In particular, if the based manifold is compact, the corresponding result
was shown in \cite[Chapter 4]{EL}.

\begin{thm}\label{t4.1}
The following log-Sobolev inequality holds for
$(\E_{\mathbf{\Lambda}},\D(\E_{\mathbf{\Lambda}}))$,
\begin{equation}\label{e16}
\mu(F^2\log F^2)\le 2 \E_{\mathbf{\Lambda}}(F,F),\ \ \ \ F \in \F C_{b,loc},
\ \mu(F^2)=1.
\end{equation}
In particular, the following weighted log-Sobolev inequality is true,
\begin{equation}\label{e16a}
\mu(F^2\log F^2)\le \int_{C_o(M)}\big(4+K(\gamma)^2\e^{-K_1(\gamma)}\big)
\|D F\|_{\H}^2 \d \mu ,\ \ \  F \in \F C_{b,loc},
\ \mu(F^2)=1,
\end{equation}
where the items $K(\gamma)$, $K_1(\gamma)$ are defined by (\ref{e15aa}).
\end{thm}
\begin{proof}
For every $R \ge 1$, let $M_R\subseteq M$, $D_R$ be the same items as that in proof of Theorem \ref{t1.1}.
Since $M_R$ has bounded Ricci curvature, by \cite{Hsu3} (see also \cite{CHL}), the integration by parts formula
holds, from which we can deduce that
the Clark-Ocone formula first developed in \cite{F} for the case $M$ compact is still true
on $C_o(M_R)$, see e.g. \cite[Chapter 4]{EL} or \cite{Hsu3}.
Hence following the same procedure as that in \cite[Section 4.2]{EL}, based on the Clark-Ocone
formula, we can show for every $R \ge 1$,
\begin{equation}\label{e17}
\mu_R(F^2\log F^2)\le 2 \int_{C_o(M_R)}
\big\|\mathbf{\Lambda}_R^{\frac{1}{2}}D_R F\big\|_{\H}^2 \d \mu_R ,\ \ \  F \in \F C_{b,loc}(M_R),
 \mu_R(F^2)=1,
\end{equation}
where $\mathbf{\Lambda}_R^{\frac{1}{2}}D_R$ is the damped gradient operator on $C_o(M_R)$, and $\mu_R$ is the Brownian measure
on $C_o(M_R)$. In particular, the constant $2$ of the log-Sobolev inequality (\ref{e17})
is independent of $R$ and the Ricci curvature bound.

Suppose $F \in \F C_{b,loc}$ with $\mu(F^2)=1$, then it has the form $F=\tilde F l(\rho)$ for some
$\tilde F \in \F C_{b}$ and $l \in C_0^{\infty}(\R)$ satisfying $\text{supp}{l}
\subseteq \{r \in \R: |r|\le R_0\}$ for some $R_0\ge 1$. Let $(\E_{R_0},\D(\E_{R_0}))$ be the
O-U Dirichlet form on $C_o(M_{R_0})$, as the same argument in
the proof of Theorem \ref{t1.1}, we know $\hat F:=F \big |_{C_o(M_{R_0})}\in \D(\E_{R_0})$ and
$$\text{supp}F = \text{supp}\hat F\subseteq \supp\{\gamma \in C_o(M_{R_0})\subseteq C_o(M):
\ \rho(\gamma)=\rho_{R_0}(\gamma)\le R_0\},$$
so $\mu_{R_0}(\hat F^2\log \hat F^2)=\mu( F^2\log  F^2)$,
$\mu_{R_0}(\hat F^2)=\mu(F^2)=1$.
As explained in the proof of Theorem \ref{t1.1}, by modeling $C_o(M)$ and
$C_o(M_{R_0})$ into the same probability space $(\Omega, \P)$, we know
$U_{\cdot,R_0}=U_{\cdot}$ $\P$ -$a.s.$ when $\hat F(\pi(U_{\cdot}))\neq 0$,
where $U_{\cdot,R_0}$ and $U_{\cdot}$ are the horizontal lift
(along the Brownian motion) on
$M_{R_0}$ and $M$ respectively, so
$\int_{C_o(M)}
\|\mathbf{\Lambda}^{\frac{1}{2}}D F\|_{\H}^2 \d \mu $
$=\int_{C_o(M_{R_0})}
\|\mathbf{\Lambda}_{R_0}^{\frac{1}{2}}D_{R_0} \hat F\|_{\H}^2 \d \mu_{R_0} $.
Then applying (\ref{e17})
to $\hat F$, we have,
\begin{equation*}
\mu(F^2\log F^2)\le 2 \int_{C_o(M)}
\|\mathbf{\Lambda}^{\frac{1}{2}}D F\|_{\H}^2 \d \mu ,
\end{equation*}
note that $F \in \F C_{b,loc}$ is arbitrary, we have shown (\ref{e16}).
And by the estimate (\ref{e15a}), we can get (\ref{e16a})
immediately from (\ref{e16}).
\end{proof}

As \cite{AE}, \cite{CHL}, \cite{Hsu2}, if the based manifold is compact
or with bounded Ricci curvature,  the log-Sobolev inequality for
the O-U Dirichlet form holds, but the corresponding constant depends on the uniform bound of
the Ricci curvature. Hence the log-Sobolev inequality may not be true if the Ricci curvature of the
based manifold is unbounded, and we will study the weak log-Sobolev inequality
introduced in \cite{CGG}, which can be used to describe the convergence rate for the entropy of
the associated Markov semigroup.

Based on the weighted log-Sobolev inequality (\ref{e16a}), following the techniques in
\cite[Lemma 2.3]{CLW} and \cite[Theorem 1]{W}, we obtain the weak log-Sobolev inequality for the O-U Dirichlet form.

\begin{thm}\label{t4.2}
If
\begin{equation}\label{e17aa}
\lim_{R \rightarrow \infty}
\frac{1}{\sqrt{\mu(\rho > R)}}\int^\infty_{R}\frac{\d s}
{\sqrt{4+\e^{-\tilde K_1(s)}\tilde K(s)^2}}=\infty,
\end{equation}
where $\tilde K_1$ and $\tilde K$ are defined by (\ref{e15c}),
then the following weak log-Sooblev inequality holds,
\begin{equation}\label{e17a}
\mu(F^2 \log F^2)\le \alpha(r) \E(F,F)+r\|F\|_{\infty}^2,\ \
F \in \F C_b \ \ \ \mu(F^2)=1,\ \ 0<r\le r_0,
\end{equation}
for some $r_0>0$ with
\begin{equation*}
\begin{split}
\alpha(r):=\inf_{R \in \Lambda_r}\big\{2\big(4+\tilde K(R)^2\e^{-\tilde K_1(R)}\big)\big\}<\infty,
\ \ r>0,
\end{split}
\end{equation*}
where
\begin{equation}\label{e18a}
\Lambda_r:=\bigg\{R>0:\ \inf_{R_1 \in (0,R)}\bigg\{\frac{2\mu(\rho > R_1)}
{\big(\int^{R}_{R_1}\frac{\d s}
{\sqrt{4+\e^{-\tilde K_1(s)}\tilde K(s)^2}}\big)^2}+
3\sqrt{\mu(\rho>R_1)}\bigg\}\le r\bigg\}.
\end{equation}
\end{thm}
\begin{proof}

 For a fixed $R_1>0$, let
 \begin{equation*}
\theta(r)=\frac{1}{\sqrt{\mu(\rho > R_1)}}\int^{R_1\vee r}_{R_1}\frac{\d s}
{\sqrt{4+\e^{-\tilde K_1(s)}\tilde K(s)^2}},\ \ r\in \R.
\end{equation*}
For every (fixed) $R>R_1$, let
\begin{equation*}
g_R(r):=\Big(1-\frac{\theta(r)}{\theta(R)}\Big)^+, \ \ r \in \R,
\end{equation*}
then $g_R$ is a bounded Lipschitz continuous function on $\R$
with compact support, moreover $0\le g_R \le 1$, $g_R(r)=1$ if $r\leq R_1$, and $g_R(r)=0$ if $r\geq R$.
So by lemma \ref{l2.1} and the approximation argument, we have $g_R(\rho)\in \D(\E)$ and
\begin{equation}\label{ee}\|D g_R(\rho)\|_{\H}^2
\le \frac{1}{(4+\e^{-\tilde K_1(\rho)}\tilde K(\rho)^2)\mu(\rho>R_1) \theta(R)^2}1_{\{\rho>R_1\}}.\end{equation}

It is sufficient to prove
(\ref{e17a}) for every $F \in \F C_{b,loc}$. For every $F \in \F C_{b,loc}$ with $\mu(F^2)=1$,
let $F_R:=g_R(\rho)F$, we have,
\begin{equation*}
\begin{split}
&\Ent(F^2)=\mu(F^2\log F^2)=
\big(\mu(F^2\log F^2)-\Ent(F_R^2)\big)+\Ent(F_R^2)\\
&=\mu\big(F^2(1-g_R^2(\rho))\log F^2\big)+
\mu\Big(F^2g_R^2(\rho)\big(-\log g_R^2(\rho)+\log \mu(F_R^2)\big)\Big)+\Ent(F_R^2)\\
&:=I_1+I_2+I_3,
\end{split}
\end{equation*}
where $\Ent(G^2):=\mu(G^2 \log G^2)-\mu(G^2)\log \mu(G^2)$ for every measurable function $G$
on $C_o(M)$.

According to the inequality $(\log x)^{+} \leq x, x>0$,
\begin{equation*}
\aligned I_1&=\int_{\{\rho>R_1\}}(1-g_R^2(\rho))F^2 \log F^2 \d \mu\\
&\leq2\|F\|_\infty^2\int_{\{\rho>R_1\}} (\log |F|)^+ \d\mu
\leq2\|F\|_\infty^2\int_{\{\rho>R_1\}} |F| \d \mu\\
&\leq2\|F\|_\infty^2 \sqrt{\mu(\rho>R_1)},\endaligned
\end{equation*}
where in the last step above, we use Cauchy-Schwarz inequality and the assumption
$\mu(F^2)=1$.

Since
$x\log x\geq -\e^{-1}$ for any $0< x\leq1$, and $\log \mu(F_R^2)\le 0$
due to $\mu(F_R^2)\le \mu(F^2) \le 1$, we have,
\begin{equation*}
I_2\leq \e^{-1}\int_{\{\rho>R_1\}} F^2 \d \mu\leq \|F\|_\infty^2\mu(\rho>R_1).
\end{equation*}

Since $\text{supp}F_R \subseteq \mathbf{B}_R$,
we apply (\ref{e16a}) to $F_R$ and obtain,
\begin{equation*}
\begin{split}
& I_3 \le
\int \big(4+K(\gamma)^2 \e^{-\tilde K_1(\gamma)}\big)\|D F_R(\gamma)\|_{\H}^2 \d \mu\\
&\le 2\int \big(4+\tilde K(\rho)^2 \e^{-\tilde K_1(\rho)}\big)\big(
g_R^2(\rho)\|D F\|_{\H}^2+ \|D g_R(\rho)\|_{\H}^2 F^2\big)\d \mu\\
&\le 2 \big(4+\tilde K(R)^2 \e^{-\tilde K_1(R)}\big)\E(F,F)
+\frac{2\|F\|_{\infty}^2}{\theta(R)^2},
\end{split}
\end{equation*}
where in the last step, we use (\ref{ee}).

Combining with the above inequalities, we obtain
$$\aligned\mu(F^2\log F^2)&\leq 2\big(4+\tilde K(R)^2 \e^{-\tilde K_1(R)}\big)\E(F,F)\\
&+\Big(\frac{2}{\theta(R)^2}+3\sqrt{\mu(\rho>R_1)}\Big)\|F\|_\infty^2,
\endaligned$$
from which we can get (\ref{e17a}) immediately. In particular, (\ref{e17aa}) ensures the set
$\Lambda_r$ defined by (\ref{e18a}) is not empty when $r$ is small enough.
\end{proof}

Given $\tilde K$ and $\tilde K_1$, i.e. the growth rate of the Ricci curvature of the based manifold, we obtain a concrete estimate
for the rate function $\alpha$ in the weak log-Sobolev inequality (\ref{e17a}). In particular, by the equivalence
between a class of weak log-Sobolev inequalities and super Poincar\'e inequalities established in \cite{CGG},
under some suitable condition of $\tilde K$ and $\tilde K_1$, we can show that
the super Poincar\'e inequality or the
Poincar\'e inequality holds for the O-U Dirichlet form. For the overall introduction of the super
Poincar\'e inequality, we refer the reader to \cite{W00a} and \cite[Chapter 3]{Wbook}.

\begin{cor}\label{cor}
Suppose
\begin{equation}\label{e18}
\tilde K(s) \le c_1(1+s^{\delta_1}),\ \ \tilde K_1(s)\ge -c_2-\delta_2 \log(1+s),\ \ \ s>0,
\end{equation}
for some non-negative constants $c_1,c_2,\delta_1,\delta_2$.

(1) If $2\delta_1+\delta_2<2$, then the following
super Poincar\'e inequality holds,
\begin{equation}\label{e19a}
\mu(F^2)\le r \E(F,F)+\beta(r)\mu(|F|)^2,\ \ \ F \in \D(\E),\ r>0,
\end{equation}
where $\beta(r)=\exp\Big(c_3\Big(1+r^{-\frac{2}{2-2\delta_1-\delta_2}}\Big)\Big)$ for some
$c_3>0$.

(2) If $2\delta_1+\delta_2\le 2$, then the following  Poincar\'e inequality
\begin{equation}\label{e20}
\mu(F^2)\le c_4\E(F,F)+\mu(F)^2,\ \ \ F \in \D(\E),
\end{equation}
holds for some $c_4>0$.
\end{cor}
\begin{proof}
Under the curvature condition (\ref{e18}), by \cite[Lemma 2.2]{W}, see also
\cite[Page 1091 (2.6)]{WW1}, we know,
\begin{equation*}
\mu(\rho>R_1)\le C_1 \e^{-C_2 R_1^2},\ \ \ R_1>0,
\end{equation*}
for some positive constants $C_1,C_2$. By (\ref{e18}),
there exist positive constants $C_3, C_4$ such that for every $R_1,R$ large enough with $R>> R_1$,
\begin{equation*}
\frac{1}{\sqrt{\mu(\rho > R_1)}}\int^{R}_{R_1}\frac{\d s}
{\sqrt{4+\e^{-\tilde K_1(s)}\tilde K(s)^2}}\ge C_3\e^{C_4 R_1^2},
\end{equation*}
so condition (\ref{e17aa}) is true. And if we take
$R_1=\frac{R}{2}$ in (\ref{e18a}), then there exist some positive constants
$C_5,C_6$, such that for any $r>0$ small enough,
$\tilde R_0:=C_5+C_6\sqrt{|\log r|}\in \Lambda_r $. So by Theorem \ref{t4.2},
(\ref{e17a}) is true with the rate function
\begin{equation}\label{e19}
\alpha(r)=C_7 |\log r|^{\frac{2\delta_1+\delta_2}{2}},\ \ r>0,
\end{equation}
for some constant $C_7>0$.

If $2\delta_1+\delta_2<2$, by \cite[Proposition 3.4]{CGG}, the weak log-Sobolev inequality with rate function
(\ref{e19}) implies the super Poincar\'e inequality (\ref{e19a}). We want to remark that although
their result is presented for the inequality on Euclidean space, by carefully tracking the proof,
such result is still  true for our case.

If  $2\delta_1+\delta_2\le 2$, by  \cite[Proposition 3.1]{CGG} (see also
\cite[Lemma 2.4]{CLW}), we obtain the Poincar\'e inequality
(\ref{e20}).
\end{proof}

\section{Quasi-regularity of the Dirichlet form}

In this section, we will study the quasi-regularity of the Dirichlet form
$(\E_{\A}, \D(\E_{\A}))$.
Let $L^{\infty}(C_o(M)\rightarrow \R^n;\mu)$ be the set of measurable vectors
$v: C_o(M) \rightarrow \R^n$ such that $||v||_{L^{\infty}}<\infty$ and let
$L^{\infty}_{loc}(C_o(M)\rightarrow \R^n;\mu)$ be the collection of
measurable vectors $v: C_o(M) \rightarrow \R^n$ such that
$v1_{\mathbf{B}_R} \in L^{\infty}(C_o(M)\rightarrow \R^n;\mu)$ for every
$R \ge 1$, where $\mathbf{B}_R$ is defined by (\ref{e6aa}). For every $t \in [0,1]$,
$v \in L^{\infty}_{loc}(C_o(M)\rightarrow \R^n;\mu)$, we define
$\Psi_{t,v}(\cdot,\gamma) \in \H$ as following
\begin{equation*}
\Psi_{t,v}(s,\gamma):= \big(s \wedge t\big)v(\gamma),\ \ \ s \in [0,1],\ \gamma \in C_o(M).
\end{equation*}
Let $l_R \in C_0^{\infty}(\R)$ be the function constructed by (\ref{e0a}), for every $R \ge 1$ we define
\begin{equation}\label{e0aa}
\phi_R(\gamma):=l_R(\rho(\gamma)),\ \ \gamma \in C_o(M).
\end{equation}
In order to prove the quasi-regularity of $(\E_{\A}, \D(\E_{\A}))$, we need to impose a condition which is stronger than
(A2). We assume
\begin{enumerate}
\item[(A2')]
For every $v \in L^{\infty}_{loc}(C_o(M)\rightarrow \R^n;\mu)$,
$\Psi_{t,v}(\cdot,\gamma) \in \D(\A(\gamma)^{\frac{1}{2}})$ $\mu$- $a.s.$.  For
every $R \ge 1$,
$D\phi_R(\gamma) \in \D(\A(\gamma)^{\frac{1}{2}})$ $\mu$-$a.s.$, and
there exists a
constant $0<c_1(R)<\infty$, such that,
\begin{equation}\label{e1a}
\begin{split}
&\int_{\mathbf{B}_R}\|\A^{\frac{1}{2}}\Psi_{t,v}\|_{\H}^2\d \mu\le
c_1(R)||v1_{\mathbf{B}_R}||^2_{L^{\infty}},
\end{split}
\end{equation}
\begin{equation}\label{e1aa}
\begin{split}
&\int_{\mathbf{B}_R}\big\|\A^{\frac{1}{2}}(D\phi_R)\big\|_{\H}^2\d \mu\le c_1(R),
\end{split}
\end{equation}
for every $v \in L^{\infty}_{loc}(C_o(M)\rightarrow \R^n;\mu)$ and $t \in (0,1]$.
\end{enumerate}
The assumption (A2') is a local version of assumption of (A1) in
\cite[Page 1087]{WW1}, and it is obvious that assumption (A2') implies (A2).

Now we will prove Theorem \ref{t1.2}, the proof is inspired by the argument
developed in \cite{DR} for path space over a compact manifold.

\begin{proof}[Proof of Theorem $\ref{t1.2}$]
 (a) We need to check (i)-(iii) in \cite[Definition IV-3.1]{MR}.
Since every $F \in \F C_{b,loc}$ is a continuous function on $C_o(M)$, and
$\F C_{b,loc}$ is dense on $\D(\E_{\A})$ under $\E_{\A,1}$ norm, so (ii) of
\cite[Definition IV-3.1]{MR} is satisfied.

Let $\{t_i\}_{i=1}^\infty$ be a countable dense subset of
$(0,1]$, $\{\phi_R\}_{R=1}^{\infty}$ be the function defined by (\ref{e0aa}), we may choose a countable dense subset $\{\eta_m\}_{m=1}^{\infty}$
$\subseteq C_0^{\infty}(M)$ which separates the points of $M$. Define
\begin{equation*}
\begin{split}
&\mathbf{S}:=\Big\{F_{R,m,i}|F_{R,m,i}(\gamma)=\phi_R(\gamma)\eta_m\big(\gamma(t_i)\big),\ \ R,i,m \in \mathbb{N}_+, \gamma\in C_o(M)\Big\},
\end{split}
\end{equation*}
it is obvious that $\mathbf{S}\subseteq \D(\E_{\A})$ is a countable subset and every $F \in \mathbf{S}$
is a continuous function. For any $\gamma, \sigma \in C_o(M)$ with $\gamma\neq\sigma$, there exists a $t_{i_0}$ such that $\gamma(t_{i_0}) \neq \sigma(t_{i_0})$, and
we can choose a $\eta_{m_0} \in C_0^{\infty}(M)$ with
$\eta_{m_0}(\gamma(t_{i_0})) \neq \eta_{m_0}(\sigma(t_{i_0}))$ since $\{\eta_m\}$ separates the point of $M$.
We can also find a
$R_0 \ge \max\{\rho(\gamma),\rho(\sigma)\}+1$,  hence $\phi_{R_0}(\gamma)=\phi_{R_0}(\sigma)=1$
by definition.
Let $\tilde F(\gamma):=\phi_{R_0}(\gamma)\eta_{m_0}\big(\gamma_{t_0}\big)$, then $\tilde F\in \mathbf{S}$ and
$\tilde F(\gamma) \neq \tilde F(\sigma)$, which implies $\mathbf{S}$ separates the point in $C_o(M)$, and
(iii) of \cite[Definition IV-3.1]{MR} is true.

In the following, it suffices to check (i) of \cite[Definition IV-3.1]{MR}, i.e. (see \cite[Remark IV-3.2]{MR})
to find out a sequence of compact sets $\{\mathbf{K}_k\}_{k=1}^{\infty}\subset C_o(M)$ such
that
\begin{equation}\label{e9}
\displaystyle\lim_{k\to\infty}\Cap(C_o(M)\backslash \mathbf{K}_k)=0,
\end{equation}
where $\Cap$ is the capacity induced by $(\E_{\A},\D(\E_{\A}))$, see e.g. \cite[Page 606]{DR}.

(b) To construct $\mathbf{K}_k$, we use the method developed in \cite{DR}, the main difference of our situation here
from that in \cite{DR} is that we can only take a local test function
$G_{j,R}$ (defined below) due to the lack of uniformly control of $\E_{\A}$ without any curvature condition.
For the reader's convenience, here we write all the procedure explicitly.
 By Nash embedding theorem,
there exists a $\varphi:M \rightarrow \R^N$ (for some $N \in \mathbb{N}_+$) such that
\begin{equation*}
M\ni x\mapsto\varphi(x):=\big(\varphi_1(x),\cdots,\varphi_N(x)\big)\in \mathbb{R}^N
\end{equation*}
is a smooth isometric embedding, i.e. $M$ is isometric to $\varphi(M)$ endowed with
the induced metric of $\mathbb{R}^N$. As introduced,
the distance $d$ defined by $d(\gamma,\sigma):=\sup_{t\in[0,1]}d_M(\gamma(t),\sigma(t))$,
$\gamma,\sigma \in C_o(M)$
is compatible with the topology on $C_o(M)$.
Let $d_0(x,y):=|\varphi(x)-\varphi(y)|$, $ x,y\in M$ and
\begin{equation}\label{e8aa}
\bar{d}(\gamma,\sigma):=\displaystyle\sup_{t\in [0,1]}d_0(\gamma(t),\sigma(t)), \ \ \
\gamma,\sigma \in C_o(M).
\end{equation}
Repeating the same procedure as that in step (1) of the proof
of \cite[Page 1092, Theorem 1.1]{WW1}, we know $\bar d$ induces the same topology on
$C_o(M)$ as $d$ does.

(c) From now on, we take $R$ to be a positive integer. Let $\{t_j\}_{j=1}^\infty$ be a countable dense subset of
$(0,1]$, let $\{\phi_R\}_{R=1}^{\infty}$ be defined by (\ref{e0aa}).
For every fixed $R,j\ge 1$ and $\sigma \in C_o(M)$, we define,
\begin{equation*}
G^j_{\sigma,R}(\gamma):=\displaystyle
\phi_R(\gamma)\big(\sup_{1\leq i\leq N}\displaystyle|
\varphi_i(\gamma(t_j))-\varphi_i(\sigma(t_j))|\wedge 1\big),
\quad \ \  \gamma \in C_o(M).
\end{equation*}
Notice that $\phi_R(\gamma)\neq 0$ only if $\gamma\in \mathbf{B}_R$, so the above definition will not be changed
if we replace $\varphi_i(\gamma(t_j))$ by $\tilde \varphi_i(\gamma(t_j))$, where
$\tilde \varphi_i \in C_0^{\infty}(M)$ with $\tilde \varphi_i(x)=\varphi_i(x)$
for each $x \in B_R$. Hence
$G^j_{\sigma,R} \in \F C_{b,loc}\subseteq \D(\E_{\A})$. From
the assumption (A2'), we have,
\begin{equation}\label{e9a}
\begin{split}
\E_{\A}(G^j_{\sigma,R},G^j_{\sigma,R}) \leq 2c_1(R)(1+c_2(R)), \ \ \ R,j\ge 1, \sigma \in C_o(M),
\end{split}
\end{equation}
where $c_1(R)$ is the constant in assumption (A2'), and $c_2(R):=\max_{1\le i \le N}
\sup_{x \in B_R}$ $|\nabla \varphi_i(x)|^2$.
For every fixed $R \ge 1$ and $\sigma \in C_o(M)$, we define
\begin{equation*}
G_{\sigma,R}(\gamma):=\phi_R(\gamma)\big(\displaystyle\sup_{1\leq i\leq N}\displaystyle\sup_{t\in[0,1]}
|\varphi_i(\gamma(t))-\varphi_i(\sigma(t))|\wedge 1\big).
\end{equation*}
By the dominated convergence theorem,
\begin{equation*}
\lim_{k \rightarrow \infty}\mu\Big(\Big|\sup_{1\le j \le k}G_{\sigma,R}^j-G_{\sigma,R}\Big|^2\Big)=0.
\end{equation*}
Since $(\E_{\A}, \D(\E_{\A}))$ is closed, by (\ref{e9a}) and according to the same argument
in the proof of Lemma \ref{l2.1}, (see also\cite[Proposition 3.1]{RS}),
 we have $G_{\sigma,R} \in \D(\E_{\A})$ and
\begin{equation}\label{e10}
\E_{\A}(G_{\sigma,R},G_{\sigma,R}) \le 2c_1(R)(1+c_2(R)), \ \ \  R\ge 1, \sigma \in C_o(M).
\end{equation}

(d) Let $\{\sigma_i\}_{i=1}^{\infty}$ be a countable dense subset of $C_o(M)$.
For any $m,R\ge 1$, let
\begin{equation*}
G_{m,R}(\gamma):=\displaystyle\inf_{1\leq i\leq m}G_{\sigma_i,R}(\gamma),\ \ \  \gamma \in C_o(M).
\end{equation*}
Similar to the above argument, we obtain from (\ref{e10}) that
$G_{m,R} \in \D(\E_{\A})$ and for every $m,R\ge 1$,
\begin{equation}\label{e10a}
\E_{\A}(G_{m,R},G_{m,R}) \le 2c_1(R)(1+c_2(R)).
\end{equation}
Since $\{\sigma_i\}_{i=1}^{\infty}\subset C_o(M)$ is dense and $\bar d$ induces the same topology as
that induced by $d$, by dominated convergence theorem we obtain for every $R\ge 1$,
\begin{equation}\label{e10a1}
\lim_{m \rightarrow \infty}\mu \big(|G_{m,R}|^2\big)=0.
\end{equation}
For a fixed $R \ge 1$, due to (\ref{e10a}), (\ref{e10a1}),
repeating the same argument in the proof of Lemma \ref{l2.1} or \cite[Proposition 3.1]{RS}
(by the Banach-Saks property),
 there exists a subsequence $\{m_i^R\}_{i=1}^{\infty}$
such that
\begin{equation}\label{e10a2}
\lim_{j \rightarrow \infty}\E_{\A}(\bar G_{j,R}, \bar G_{j,R})=0,
\end{equation}
where
$\bar G_{j,R}:=\frac{1}{j}\sum_{i=1}^j G_{m_i^R,R}$. We can also find a subsequence
$\{j^R\}_{j=1}^{\infty}$ $\subseteq \{j\}_{j=1}^{\infty}$, such that for every $j$,
\begin{equation}\label{e11}
\begin{split}
&\E_{\A,1}\big(\bar G_{(j+1)^R,R}-\bar G_{j^R,R},\bar G_{(j+1)^R,R}-\bar G_{j^R,R}\big)\\
&:=\E_{\A}\big(\bar G_{(j+1)^R,R}-\bar G_{j^R,R},\bar G_{(j+1)^R,R}-\bar G_{j^R,R}\big)+
\|\bar G_{(j+1)^R,R}-\bar G_{j^R,R}\|_{L^2(\mu)}^2 \le 2^{-5(j+R)}.
\end{split}
\end{equation}
Since (\ref{e10a}) holds for each $m=m_{j^R}^R$, as the same way above, we can find subsequence
$\{m_{j^R}^{R}\}_{j=1}^{\infty}$ and
$\{j^{R+1}\}_{j=1}^{\infty}$, such that $\{m_i^{R+1}\}_{i=1}^{\infty}$ is a subsequence of
$\{m_{j^R}^{R}\}_{j=1}^{\infty} $,
and for every $j$,
\begin{equation*}
\E_{\A,1}\big(\bar G_{(j+1)^{R+1},R+1}-\bar G_{j^{R+1},R+1},\bar G_{(j+1)^{R+1}+1,R+1}-\bar G_{j^{R+1},R+1}\big)
\le 2^{-5(j+R+1)},
\end{equation*}
where $\bar G_{j,R+1}:=\frac{1}{j}\sum_{i=1}^j G_{m_i^{R+1},R+1}$ for every positive integer $j$.
Then by induction, for every $R\ge 1$, we can construct subsequence
$\{m_{i}^{R}\}_{i=1}^{\infty} $ and $\{j^R\}_{j=1}^{\infty}$, such that
$\{m_i^{R+1}\}_{i=1}^{\infty}$  is a subsequence of $\{m_{j^R}^{R}\}_{j=1}^{\infty} $ and (\ref{e11})
is true for every $j,R$.


Let $Y_{j,R}:=\{\gamma \in C_o(M): \ \bar G_{(j+1)^R,R}-\bar G_{j^R,R}>2^{-j}\}$, since
$\bar G_{j,R}$ is continuous by definition (note that we have shown $d$ and
$\bar d$ induce the same topology on $C_0(M)$), $Y_{j,R}$ is a open set, hence by (\ref{e11})
and the same argument for (2.1.10) in the proof of \cite[Theorem 2.1.3]{FOT}, we have,
\begin{equation}\label{e11a}
\Cap(Y_{j,R})\le \frac{\E_{\A,1}\big(\bar G_{(j+1)^R,R}-\bar G_{j^R,R},\bar G_{(j+1)^R,R}-\bar G_{j^R,R}
\big)}{2^{-2j}}\le 2^{-3j-5R}.
\end{equation}
For every $k\ge 1$, let
\begin{equation*}
Z_k:=\bigcup_{R=1}^{\infty}\bigcup_{j=k}^{\infty}Y_{j,R},\ \ \ \mathbf{K}_k:=C_o(M)\backslash Z_k,
\end{equation*}
so $\mathbf{K}_k$ is closed, and  for any $R \ge 1$,
\begin{equation}\label{e12}
\bar G_{j^R,R}(\gamma)\le 2^{-j+1},\ \  \gamma \in \mathbf{K}_k, \ j\ge k.
\end{equation}
Note that $G_{j+1,R}\le G_{j,R}$, so $G_{m_{j^R}^R,R} \le\bar G_{j^R,R}$, and
for any $R \ge 1$, (\ref{e12}) holds for $G_{m_{j^R}^R,R}$.
By construction above, $\{m_i^{R+1}\}_{i=1}^{\infty}$, hence $\{m_{j^{R+1}}^{R+1}\}_{j=1}^{\infty}$, is a subsequence
of $\{m_{j^R}^{R}\}_{j=1}^{\infty}$, by choosing the diagonal subsequence, we can find a
subsequence $\{q_j\}_{j=1}^{\infty}$, such that for every $R\ge 1$,
\begin{equation*}
G_{q_j,R}(\gamma)\le 2^{-j+1},\ \  \gamma \in \mathbf{K}_k, \ j\ge k.
\end{equation*}
By the definition of $G_{q_j,R}$, let $R \to \infty$,
note that $\phi_R \to 1$ as $R \to \infty$, we obtain,
\begin{equation}\label{e12a}
\inf_{1\le r \le q_j}\sup_{1\le i \le N}\sup_{t \in [0,1]}
\big|\varphi_i(\gamma(t))-\varphi_i(\sigma_r(t))\big|\le 2^{-j+1},\ \  \gamma \in \mathbf{K}_k, \ j\ge k.
\end{equation}
It is obvious that (\ref{e12a}) implies that $\mathbf{K}_k$ is totally bounded with respect to
the metric $\bar d$ defined by (\ref{e8aa}), also note that $\mathbf{K}_k$ is closed and
the topology  on $C_o(M)$ induced by $d$ and $\bar d$ is the same, we know
$\mathbf{K}_k$ is compact.

On the other hand, by (\ref{e11a}) and \cite[Lemma 2.3]{RS},
\begin{equation*}
\begin{split}
& \Cap\big(C_o(M)\backslash \mathbf{K}_k\big)=\Cap\big(\bigcup_{R=1}^{\infty}\bigcup_{j=k}^{\infty}Y_{j,R}\big)\\
&\le\sum_{R=1}^{\infty}\sum_{j=k}^{\infty}\Cap(Y_{j,R}) \le \sum_{R=1}^{\infty}\sum_{j=k}^{\infty}2^{-3j-5R}\le \frac{8\cdot 2^{-3k}}{217},
\end{split}
\end{equation*}
which implies that (\ref{e9}) is true. So by now we have completed the proof.
\end{proof}

\begin{Remark}
Here $\A$ can also be viewed as an (pointwise defined) operator from
$L^{\infty}(C_o(M)\rightarrow \H;\mu)$ to $L^{2}(C_o(M)\rightarrow \H;\mu)$ such that
\begin{equation}\label{e0}
\big(\A^{\frac{1}{2}}\big)\Phi(\gamma)=\A(\gamma)^{\frac{1}{2}}\Phi(\gamma),\ \mu-a.s. \  \gamma \in C_o(M),
\Phi \in \D(\A).
\end{equation}
In \cite{WW1}, with some restriction on the curvature of the based manifold, the
closability and quasi-regularity was shown for
$(\E_{\A}, \D(\E_{\A}))$ without the condition (\ref{e0}) on $\A$.
\end{Remark}

Repeating the proof of \cite[Proposition 5 (ii)]{DR} or \cite[Proposition 3.4]{L}, we can show
the the locality of $(\E_{\A}, \D(\E_{\A}))$,
\begin{prp}\label{p3.1} Suppose assumption (A1), (A2') and (A3) hold, the Dirichlet
form $(\E_{\A}, \D(\E_{\A}))$ is local.
\end{prp}

By \cite[Theorem IV 3.5]{MR} and \cite[Proposition V 1.11]{MR}, Theorem \ref{t1.2} and Proposition \ref{p3.1}, we get the following the result,

\begin{thm}\label{t3.2} Suppose assumption (A1), (A2') and (A3) hold,
then there exists a diffusion process $\mathbf{M}=\big(\Omega, \scr{F}, (\scr{F}_t)_{t \ge 0}, \xi_t,
(\mathbb{P}_z)_{z \in C_o(M)}\big)$ associated with $(\E_{\A}, \D(\E_{\A}))$, i.e.
$\mathbf{M}$ is a strong Markov process with continuous trajectories, and
for every $t>0$, $u \in L^2(\mu)$ bounded,
\begin{equation*}
T_t u(z)=\int u(\xi_t)\d \P_z,\ \ \mu-a.s.\ z \in C_o(M),
\end{equation*}
where $T_t$ denotes the $L^2(\mu)$ semigroup associated with $(\E_{\A}, \D(\E_{\A}))$.
\end{thm}


If $\A=\mathbf{I}$, $(\E_{\A}, \D(\E_{\A}))$ is the O-U Dirichlet form, and assumption
(A1), (A2'), (A3) hold, so we get the following corollary,

\begin{cor}
The O-U Dirichlet form $(\E, \D(\E))$ is quasi-regular.
\end{cor}

Moreover, it is easy to check for the damped Dirichlet form in Example \ref{ex1},
(A2') is true. Since (A1) and (A3) are verified in Example \ref{ex1},
the damped Dirichlet form is also quasi-regular.

We provide the following example such that $\A(\gamma)$ may be an unbounded operator, which can be viewed
a generalization of that in \cite{L} and \cite{WW1}.

\begin{exa}
\end{exa}

We first introduce an orthonormal basis $\{H_m\}_{m=1}^{\infty}$
of $\H$, which is constructed in \cite[Page 3]{L}. Let $S_1\equiv 1$ and
\begin{equation*}
S_{2^k+i}(s):=\begin{cases}
&2^{\frac{k}{2}},\ \ \ \ \ \ ~~\text{if}\ s \in [(i-1)2^{-k}, (2i-1)2^{-(k+1)}),\\
&-2^{\frac{k}{2}},\ \ \ \ \ \text{if}\ s \in [(2i-1)2^{-(k+1)}, i 2^{-k}),\\
&0,\ \ \ \ \ \ \ \ ~~\text{otherwise},
\end{cases}
\end{equation*}
for every integer $k\ge 0$ and $1\le i \le 2^k$. Let
\begin{equation*}
H_{n(p-1)+j}(t):=\int_0^t S_p(s)e_j ds,\ \ \ p,j \in \mathbb{N}_+~\text{with}\ \ 1\le j \le n,
\end{equation*}
where $\{e_j\}_{j=1}^n$ is an orthonormal basis of $\R^n$. Then $\{H_m\}_{m=1}^{\infty}$ is an orthonormal
basis of $\H$.

Let $\D(\A(\gamma))\subseteq \H$ be the domain of $\A(\gamma)$. We
suppose $\{H_m\}_{m=1}^{\infty}\subseteq \D(\A(\gamma))$ for $\mu$-$a.s.$ $\gamma \in C_o(M)$ and
for every $m \ge 1$,
$\A(\gamma)H_m=\lambda_m(\gamma)H_m$, where $\lambda_m(\cdot)\ge 0$
is a measurable function on $C_o(M)$. Hence
\begin{equation*}
\D(\A(\gamma)^{\frac{1}{2}})=\bigg\{h=\sum_{m=1}^{\infty}\langle h, H_m\rangle_{\H}H_m:\
\sum_{m=1}^{\infty}\lambda_m(\gamma)\big|\langle h, H_m\rangle_{\H}\big|^2<\infty \bigg\},
\end{equation*}
and
\begin{equation*}
\A(\gamma)^{\frac{1}{2}}h=\sum_{m=1}^{\infty}\lambda_m^{\frac{1}{2}}(\gamma)\langle h, H_m\rangle_{\H}H_m,\ \ \
\ h \in \D(\A(\gamma)^{\frac{1}{2}}),
\end{equation*}
by which we can check easily that assumption (A1) is true. We also assume that for every
$R \ge 1$,
\begin{equation}\label{e13aa}
\lambda_m(\gamma)\ge \vv(R),\ \ \ \gamma \in \mathbf{B}_R,\ \ m \ge 1,
\end{equation}
for some constant $\vv(R)>0$.
So assumption (A3) for $\A$ is true.

By the same computation in
\cite[Page 1089]{WW1}, we obtain for every $t \in [0,1]$,
$v \in L_{loc}^{\infty}(C_o(M)\to \H; \mu)$,
\begin{equation}\label{e13a}
\begin{split}
& \big\|\A(\gamma)^{\frac{1}{2}}\big((t \wedge \cdot)v(\gamma)\big)\big\|^2_{\H}\\
&\le\Big\{\sum_{j=1}^n \lambda_j(\gamma)+\sum_{k=0}^{\infty}\sum_{i=1}^{2^k}
\sum_{j=1}^n\lambda_{n(2^k+i-1)+j}(\gamma)2^{-k}
1_{\{((i-1)2^{-k},i2^{-k})\}}(t)\Big\}|v(\gamma)|^2,
\end{split}
\end{equation}
so if we assume for every $R\ge 1$,
\begin{equation}\label{e13}
\sum_{j=1}^n\mu\big(\lambda_j 1_{\mathbf{B}_R}\big)+\sum_{k=0}^{\infty}\sum_{i=1}^{2^k}
\sum_{j=1}^n\mu\big(\lambda_{n(2^k+i-1)+j}
1_{\mathbf{B}_R}\big)2^{-k}<\infty,
\end{equation}
then $(t \wedge \cdot)v(\gamma) \in \D(\A(\gamma)^{\frac{1}{2}})$ $\mu$-$a.s.$, and (\ref{e1a}) in (A2') holds.

Next we are going to check that (\ref{e1aa}) is true under (\ref{e13}).
Let $\phi_R^m(\gamma):=l_R(\rho^m(\gamma)),~\gamma\in C_o(M)$, where $l_R\in C_0^\infty(\mathbb{R})$
is constructed by (\ref{e0a}) and $\rho^m$ is defined by (\ref{c*}).
Since $\mu\Big(d_M(o,\gamma(t_i))=d_M(o,\gamma(t_j))\Big)=0$ for every $i \neq j$, we have for
each $s \in (0,1]$,
\begin{equation*}
\big(D\phi_R^m(\gamma)\big)(s)=\sum_{i=1}^m (t_i \wedge s)
v_i(\gamma)1_{\{\rho^m(\gamma)=
d_M(\gamma(t_i),o)\}},\ \mu-a.s. \gamma\in C_o(M),
\end{equation*}
where $v_i(\gamma):=l_R'(\rho^m(\gamma))U_{t_i}(\gamma)^{-1}\big(\nabla d_M(\gamma(t_i),o) \big)$. Note that
$v_i \in L_{loc}^{\infty}(C_o(M) \to \R^d; \mu)$,
by (\ref{e13a}) we obtain,
\begin{equation*}
\begin{split}
&\big\|\A(\gamma)^{\frac{1}{2}}(D\phi_R^m(\gamma))\big\|^2_{\H}=\sum_{i=1}^m
\big\|\A(\gamma)^{\frac{1}{2}}\big((t_i \wedge \cdot)v_i(\gamma)\big) \big\|_{\H}^2
1_{\{\rho^m(\gamma)=
d_M(\gamma(t_i),o)\}}\\
&\le 2\sum_{j=1}^n \lambda_j(\gamma)+2\sum_{k=0}^{\infty}\sum_{i=1}^{2^k}
\sum_{j=1}^n\lambda_{n(2^k+i-1)+j}(\gamma)2^{-k},\ \ \mu-a.s.,
\end{split}
\end{equation*}
where we use $\|v_i\|_{L^{\infty}}\le 2$ and for $\mu-a.s.$ $\gamma$, there is only
one $1\le i \le m$, such that $1_{\{\rho^m(\gamma)=
d_M(\gamma(t_i),o)\}} \neq 0$.   So according to (\ref{e13}),
\begin{equation*}
\sup_{m} \big\|\A(\gamma)^{\frac{1}{2}}(D\phi_R^m(\gamma))\big\|^2_{\H}
<\infty,\ \ \mu-a.s..
\end{equation*}
Since $\A(\gamma)^{\frac{1}{2}}$ is closed, based on this
and as the same argument in Lemma \ref{l2.1}, by the Banach-Saks property,
we have, $D \phi_R \in \D(\A(\gamma)^{\frac{1}{2}})$, and
\begin{equation*}
\begin{split}
& \big\|\A(\gamma)^{\frac{1}{2}}(D\phi_R(\gamma))\big\|^2_{\H}
\le 2\sum_{j=1}^n \lambda_j(\gamma)+2\sum_{k=0}^{\infty}\sum_{i=1}^{2^k}
\sum_{j=1}^n\lambda_{n(2^k+i-1)+j}(\gamma)2^{-k},\ \mu-a.s,
\end{split}
\end{equation*}
applying (\ref{e13}) to this, we get (\ref{e1aa}).

Coming all the analysis above, we know (A1), (A2') and (A3) are true if
(\ref{e13aa}) and (\ref{e13}) holds. Moreover, it is easy to see that if for every
$R\ge 1$, there are positive constants $c(R),\delta(R)$ such that,
\begin{equation*}
\mu\big(\lambda_m 1_{\{\rho \le R\}}\big)\le c(R)m^{1-\delta(R)},\ \ \ m\ge 1,
\end{equation*}
then (\ref{e13}) holds.

\section{Appendix}
As in the proof of Theorem \ref{t1.1}, for every $R\ge 1$, we take a
non-negative function $f_R \in C_0^{\infty}(M)$, such that
$f \big |_{B_R}=1$, and $M_R:=\{x \in M:\ f_R(x)>0\}$ is a connected set.
Then we define a metric $\langle, \rangle_R$ on $M_R$ as
\begin{equation*}
\langle, \rangle_R:=f_R^{-2} \langle, \rangle,
\end{equation*}
where $\langle, \rangle$ is the Riemannian metric on $M$.

Step by step following the argument in \cite[Section 2]{TW}, we
know $(M_R, \langle, \rangle_R)$ is a complete Riemannian manifold. But it was not written as a Lemma in
\cite[Section 2]{TW} for such conclusion, so for convenience of the reader, we include the following lemma
here.

\begin{lem}\label{l5.1}(\cite{TW})
$(M_R, \langle, \rangle_R)$ is a complete Riemannian manifold.
\end{lem}
\begin{proof}
Let $d_R$ be the Riemannian distance on $(M_R, \langle, \rangle_R)$, to prove the conclusion, it suffices to
show that for any $d_R$-Cauchy sequence $\{x_i\}_{i=1}^{\infty}\subseteq M_R$, there is a limit point
$x_0 \in M_R$, i.e. $\lim_{i \rightarrow \infty}d_R(x_i,x_0)=0$.

For any $m\ge 1$, let $O_m:=\{x \in M: f_R(x)>\frac{1}{m}\ \}$. If there exists a $m_0>0$, such that
$\{x_i\}_{i=1}^{\infty}\subseteq O_{m_0}$, since $O_{m_0}$ is relatively compact and $(M,\langle,\rangle)$
is complete, there exist a $x_0 \in O_{m_0+1}$ and a subsequence $\{x_{i_k}\}_{k=1}^{\infty}$, such that
$\lim_{k \rightarrow \infty}d_M(x_{i_k},x_0)=0$, where $d_M$
denotes the Riemannian distance on $(M,\langle,\rangle)$.
By the continuity of $f_R$,
when $k$ is big enough, for the minimal geodesic $\gamma_k(\cdot)$ on $(M,\langle,\rangle)$
connecting $x_{i_k}$ and $x_0$, we have $f_R(\gamma_k(t))\ge \frac{1}{m_0+2}$ for all $t$,
so by definition $d_R(x_{i_k},x_0)\le (m_0+2) d_M(x_{i_k},x_0)$, hence
$\lim_{k \rightarrow \infty}d_R(x_{i_k},x_0)=0$. Note that a Cauchy sequence is convergent if a
subsequence is convergent, so we obtain $\lim_{i \rightarrow \infty}d_R(x_i,x_0)=0$.

If $\{x_i\}_{i=1}^{\infty}\nsubseteq O_{m}$ for any $m \ge 1$, by selecting a subsequence
if necessary, we can
assume that $\lim_{i \rightarrow \infty}d_M(x_i,x_0)=0$, for some $x_0 \in M$ satisfying
$f_R(x_0)=0$ and there exists a subsequence $\{m_i\}_{i=1}^{\infty}$ of $\{m\}_{m=1}^{\infty}$, such that,
$x_i \notin O_{m_i}$, $x_i \in O_{m_{i+1}}$, $m_{i+2}>2m_{i+1}$, $\lim_{i \rightarrow \infty}m_i=\infty$.

On the other hand, suppose $\gamma \in C([0,1];M)$ is a $C^1$ path such that
$\gamma(0)=x_i$ and $\gamma(1)=x_{i+2}$, let
\begin{equation*}
\tau:=\sup\{t \in [0,1]: \ \gamma(t) \in O_{m_{i+1}}\},
\end{equation*}
so $0<\tau<1$ since $x_i \in O_{m_{i+1}}$, $x_{i+2} \notin O_{m_{i+1}}$, and
$f_R(\gamma(t))\le \frac{1}{m_{i+1}}$ for every $t \in  [\tau,1]$.
Hence we have,
\begin{equation*}
\begin{split}
& \int_0^1 |\gamma'(t)|_R \d t\ge \int_{\tau}^1 |\gamma'(t)|_R \d t
\ge m_{i+1}\int_{ \tau}^1 |\gamma'(t)| \d t\ge m_{i+1} d_M(O_{m_{i+1}},O^c_{m_{i+2}}),
\end{split}
\end{equation*}
note that $\gamma$ is arbitrary, we obtain,
\begin{equation*}
d_R(x_i,x_{i+2})\ge m_{i+1} d_M(O_{m_{i+1}},O^c_{m_{i+2}}).
\end{equation*}
Since $|f_R(x)-f_R(y)|\ge \frac{1}{m_{i+1}}-\frac{1}{m_{i+2}}$ for all $ x \in O_{m_{i+1}}$,
$y \in O^c_{m_{i+2}}$, by mean value theorem,
\begin{equation*}
\begin{split}
&\frac{1}{m_{i+1}}-\frac{1}{m_{i+2}}\le |f_R(x)-f_R(y)|\le C_1d_M(x,y),
\ \ \ x \in O_{m_{i+1}},\ y \in O^c_{m_{i+2}},
\end{split}
\end{equation*}
where $C_1:=\sup_{x \in M}|\nabla f(x)|>0$, which implies that
$d_M(O_{m_{i+1}},O^c_{m_{i+2}})\ge $ $\frac{1}{C_1} \big(\frac{1}{m_{i+1}}-\frac{1}{m_{i+2}}\big)$.
So we have,
\begin{equation}\label{e20a}
\begin{split}
&d_R(x_i,x_{i+2})\ge m_{i+1}d_M(O_{m_{i+1}},O^c_{m_{i+2}})\\
&\ge \frac{m_{i+1}}{C_1}\big(\frac{1}{m_{i+1}}-\frac{1}{m_{i+2}}\big)\ge \frac{1}{2C_1},
\end{split}
\end{equation}
where we use $m_{i+2}>2m_{i+1}$ in the last step. It is obvious that
(\ref{e20a}) contradicts with the assumption $\{x_i\}_{i=1}^{\infty}$ is a $d_R$-Cauchy sequence, so
we must have $\{x_i\}_{i=1}^{\infty}\subseteq O_{m_0}$ for some $m_0 \ge 1$, and as the analysis above,
there exists a limit point $x_0 \in M_R$ for $\{x_i\}_{i=1}^{\infty}$.
\end{proof}

\section*{Acknowledgments}

We would like to thank Professor Feng-Yu Wang for useful conversations. This research is  supported in part by
the FCT(PTDC/MAT/104173/2008) and Specialized Research Fund for the Doctoral Program of Higher Education (No. 20120071120001) of China

\beg{thebibliography}{99}

\leftskip=-2mm
\parskip=-1mm

\bibitem{A} S. Aida, \emph{Logarithmic derivatives of heat kernels and logarithmic Sobolev inequalities
with unbounded diffusion coefficients on loop spaces,} J. Funct.
Anal. 174(2000), 430--477.

\bibitem{AE} S. Aida and K. D. Elworthy, \emph{Differential calculus on path and loop spaces. I. Logarithmic
Sobolev inequalities on path spaces,} C. R. Acad. Sci. Paris S\'erie I, 321(1995), 97--102.

\bibitem{BLW}  D. Bakry, M. Ledoux, F.-Y. Wang, \emph{Perturbations of functional inequalities using growth conditions,}
J. Math. Pures Appl. 87(2007), 394--407.



\bibitem{CHL} B. Capitaine,  E. P. Hsu and M. Ledoux, \emph{Martingale representation and a simple proof of logarithmic
Sobolev inequalities on path spaces,} Electron. Comm. Probab. 2(1997),
71--81.

\bibitem{CGG} P. Cattiaux, I. Gentil and A. Guillin,
\emph{Weak logarithmic Sobolev inequalities and entropic convergence,}
Probab. Theory. Related. Fields. 139 (2007), 563--603.

\bibitem{CLW} X. Chen, X.-M. Li and B. Wu, \emph{A Poincar\'e inequality on loop spaces,}
J. Funct. Anal. 259 (2010), 1421--1442.

\bibitem{CM} A. B. Cruzeiro and P. Malliavin, \emph{Renormalized differential geometry
and path space: structural equation, curvature,}  J. Funct. Anal. 139: 1 (1996), 119--181.

\bibitem{D1} B. K. Driver, \emph{A Cameron-Martin type quasi-invariance theorem for Brownian
motion on a compact Riemannian manifolds,} J. Funct. Anal.
110(1992), 273--376.


\bibitem{DR} B. K. Driver and M. R\"{o}ckner, \emph{Construction of diffusions on path and
loop spaces of compact Riemannian manifolds,} C. R. Acad. Sci. Paris
S\'eries I 315(1992), 603--608.

\bibitem{EL} K. D. Elworthy, X.- M. Li and Y. Lejan, \emph{On The geometry of diffusion
operators and stochastic flows,} Lecture Notes in Mathematics, 1720(1999), Springer-Verlag.

\bibitem{EM} K. D. Elworthy and Z.-M. Ma, \emph{Vector fields on mapping spaces and
related Dirichlet forms and diffusions,} Osaka. J. Math. 34(1997), 629--651.

\bibitem{ES} O. Enchev and D. W. Stroock, \emph{Towards a Riemannian geometry on the path space over a Riemannian manifold,}
 J. Funct. Anal. 134 : 2 (1995), 392--416.

\bibitem{F} S.- Z. Fang, \emph{Un in\'equalit\'e du type Poincar\'esur un espace de chemins,}
C. R. Acad. Sci. Paris S\'erie I 318(1994), 257--260.

\bibitem{FM} S.- Z. Fang and P. Malliavin, \emph{Stochastic analysis on the path
space of a Riemannian manifold: I. Markovian stochastic calculus,} J. Funct. Anal. 118 : 1 (1993), 249--274.

\bibitem{FWW} S.- Z. Fang, F.-Y. Wang and B. Wu, \emph{Transportation-cost
inequality on path spaces with uniform distance,} Stochastic. Process. Appl. 118: 12 (2008), 2181--2197.

\bibitem{FOT} M. Fukushima,  Y. Oshima and M. Takeda,
\emph{Dirichlet Forms and Symmetric Markov Processes,} Walter de
Gruyter, 2010.

\bibitem{Hsu1} E. P. Hsu, \emph{Quasi-invariance of the Wiener measure on the path space over a compact
Riemannian manifold,} J. Funct. Anal. 134(1995), 417--450.

\bibitem{Hsu4} E. P. Hsu, \emph{Logarithmic Sobolev inequalites on path spaces over compact Riemannian manifolds,}
Commun. Math. Phys. 189(1997), 9?16.

\bibitem{Hsu3} E. P. Hsu, \emph{Analysis on path and loop spaces,} in "Probability Theory and Applications"
(E. P. Hsu and S. R. S. Varadhan, Eds.), LAS/PARK CITY Mathematics Series, 6(1999), 279--347,
Amer. Math. Soc. Providence.

\bibitem{Hsu2} E. P. Hsu, \emph{Stochastic Analysis on Manifold,} American Mathematical Society, 2002.

\bibitem{HO} E. P. Hsu and C. Ouyang, \emph{Cameron-Martin theorem for complete Riemannian manifolds,}
J. Funct. Anal. 257: 5 (2009), 1379--1395.


\bibitem{L} J.-U. L\"{o}bus, \emph{A class of processes on the path space
over a compact Riemannian manifold with unbounded diffusion,} Tran.
Ame. Math. Soc. (2004), 1--17.

\bibitem{MR} Z.- M. Ma and M. R\"{o}ckner, \emph{Introduction to the
theory of (non-symmetric) Dirichlet forms,} Berlin: Springer, 1992.

\bibitem{RS} M. R\"{o}ckner and B. Schmuland, \emph{Tightness of general $C_{1,p}$ capacities on Banach
space,} J. Funct. Anal. 108(1992), 1--12.

\bibitem{TW} A. Thalmaier and F.-Y. Wang, \emph{Gradient estimates for
 harmonic functions on regular domains in Riemannian manifolds,} J. Funct. Anal. 155:1(1998),109--124.

\bibitem{W00a}
F.-Y. Wang, \emph{Functional inequalities for empty essential spectrum,}
 J.\ Funct.\ Anal. 170(2000), 219--245.

\bibitem{W} F.- Y. Wang, \emph{Weak poincar\'{e} Inequalities on path
spaces,} Int. Math. Res. Not. 2004(2004), 90--108.

\bibitem{Wbook}
F.-Y. Wang,   \emph{Functional Inequalities, Markov Processes and
Spectral Theory,} Science Press, Beijing, (2005).

\bibitem{WW1} F.- Y. Wang and B. Wu, \emph{Quasi-Regular Dirichlet Forms on
Riemannian Path and Loop Spaces,} Forum Math. 20(2008), 1085--1096.

\bibitem{WW2} F. -Y. Wang and B. Wu, \emph{Quasi-Regular Dirichlet Forms on Free Riemannian Path and Loop Spaces,} Inf. Dimen. Anal. Quantum Probab.
and Rel. Topics 2(2009) 251--267.

\end{thebibliography}
\end{document}